\documentclass[11pt]{article}
\usepackage{mathrsfs}
\usepackage{amsfonts}
\usepackage[titletoc]{appendix}
\usepackage[leqno]{amsmath}
\usepackage{graphicx}
\usepackage{latexsym}
\usepackage{amsmath, amsfonts, amssymb, amsthm, mathrsfs, euscript, makeidx, color}
\usepackage{enumerate}
\usepackage{indentfirst}
\usepackage[numbers,sort&compress]{natbib}

\def\sqr#1#2{{\vcenter{\vbox{\hrule height.#2pt
    \hbox{\vrule width.#2pt height#1pt \kern#1pt \vrule width.#2pt}
    \hrule height.#2pt}}}}

\def\3n{\negthinspace \negthinspace \negthinspace }
\def\2n{\negthinspace \negthinspace }
\def\1n{\negthinspace }
\def\bel{\begin{equation}\label}
\def\eel{\end{equation}}

\def\dbF{\mathbb{F}}

\def\dbR{\mathbb{R}}
\def\dbS{\mathbb{S}}

\def\={\buildrel \triangle \over =}

\def\ds{\displaystyle}

\def\ns{\noalign{\ss}}
\def\a{\alpha}

\def\d{\delta}
\def\e{\varepsilon}

\def\k{\kappa}

\def\f{\varphi}
\def\th{\theta}

\def\i{\infty}
\def\Th{\Theta}

\def\cE{{\cal E}}
\def\cF{{\cal F}}

\def\cK{{\cal K}}
\def\cL{{\cal L}}

\def\no{\noindent}

\def\ss{\smallskip}
\def\ms{\medskip}

\def\q{\quad}
\def\qq{\qquad}

\def\ti{\tilde}
\def\cd{\cdot}

\def\diag{\hbox{\rm $\, $diag$\, $}}

\def\tr{\hbox{\rm tr$\, $}}

\def\({\Big (}
\def\){\Big )}
\def\[{\Big[}
\def\]{\Big]}
\def\bde{\begin{definition}\label}
\def\ede{\end{definition}}
\def\be{\begin{equation}}
\def\bel{\begin{equation}\label}
\def\ee{\end{equation}}
\def\bt{\begin{theorem}\label}
\def\et{\end{theorem}}
\def\bc{\begin{corollary}\label}
\def\ec{\end{corollary}}
\def\bl{\begin{lemma}\label}
\def\el{\end{lemma}}
\def\bp{\begin{proposition}\label}
\def\ep{\end{proposition}}
\def\bas{\begin{assumption}}
\def\eas{\end{assumption}}
\def\br{\begin{remark}\label}
\def\er{\end{remark}}
\def\ba{\begin{array}}
\def\ea{\end{array}}
\def\ed{\end{document}}

\def\square#1{\vbox{\hrule\hbox{\vrule height#1%
  \kern#1\vrule}\hrule}}
\def\rectangle#1#2{\vbox{\hrule\hbox{\vrule height#1%
  \kern#2\vrule}\hrule}}

\font\tenbb=msbm10 \font\sevenbb=msbm7 \font\fivebb=msbm5

\newfam\bbfam
\scriptscriptfont\bbfam=\fivebb \textfont\bbfam=\tenbb
\scriptfont\bbfam=\sevenbb

\newtheorem{theorem}{Theorem}[section]
\newtheorem{corollary}[theorem]{Corollary}

\newtheorem{lemma}[theorem]{Lemma}
\newtheorem{proposition}[theorem]{Proposition}

\theoremstyle{definition}
\newtheorem{definition}[theorem]{Definition}
\newtheorem{remark}[theorem]{Remark}

 \newtheorem{example}{Example}[section]

\makeatletter
 
 \@addtoreset{equation}{section}
\makeatother

\begin{document}

\title{Infinite Horizon Mean-Field  Linear Quadratic  Optimal Control Problems  with Jumps  and the related Hamiltonian Systems\footnote{This work is supported by the Natural Science Foundation of Jilin Province for Outstanding Young Talents
(No. 20230101365JC)  and National Natural Science Foundation of China (12271304, 11971099, 12371443).}}

\author{Qingmeng Wei \footnote{School of Mathematics and Statistics, Northeast Normal University, Changchun 130024, P. R. China; {E-mail: weiqm100@nenu.edu.cn}}\qq
Yaqi Xu \footnote{School of Mathematics and Statistics, Northeast Normal University, Changchun 130024, P. R. China; {E-mail: xuyq222@nenu.edu.cn}}\qq
Zhiyong Yu\footnote{School of Mathematics, Shandong University, Jinan 250100, P. R. China; {E-mail: yuzhiyong@sdu.edu.cn}}}

\maketitle

\begin{abstract} In this work, we focus on an  infinite horizon mean-field linear-quadratic stochastic control problem with jumps. Firstly, the infinite horizon linear mean-field stochastic differential equations and backward stochastic differential equations
with jumps are studied  to support the research of the control problem. The global integrability properties of their solution processes are studied by introducing a kind of so-called dissipation conditions suitable  for  the systems involving the mean-field terms and jumps.
 For the control problem, we conclude  a sufficient and necessary condition   of   open-loop optimal control    by the variational approach.
 Besides, a kind of infinite horizon fully coupled linear mean-field forward-backward stochastic differential equations with jumps is   studied by using the method of continuation. Such a research makes the characterization of the open-loop optimal controls   more straightforward and complete.

 \end{abstract}

\ms

\no\bf Keywords: \rm infinite horizon, mean-field FBSDE with jumps,   linear-quadratic, open-loop control


\no\bf AMS Mathematics Subject Classification. \rm
93E20, 60H10, 49N10

\section{Introduction}\label{SEC_Int}

This paper is mainly about a kind of   infinite horizon  linear quadratic  (LQ, for short) stochastic optimal control problems.
We firstly introduce   the framework which   will be worked on. Let $(\Omega, \mathcal{F}, \mathbb{F}, \mathbb{P})$ be a complete filtered probability space, where the filtration $\mathbb{F}=\left\{\mathcal{F}_{t} ; 0 \leq t<\infty\right\}$ is generated by the following mutually independent processes and augmented by all $\mathbb{P}$-null sets,

(i) a $d$-dimensional standard Brownian motion $W(\cdot)=(W_1(\cdot),\dots,W_d(\cdot))^\top$;

(ii) a Poisson random measure $\mu=(\mu_1,\dots,\mu_l)^\top$ defined on $\mathbb{R}^{+} \times E$ with $\mu_i,$ $i=1,\cdots l$ being independent and   the corresponding compensator $\hat{\mu}(\mathrm dt, \mathrm de)=(\hat{\mu}_1(\mathrm dt, \mathrm de),\dots,\hat{\mu}_l(\mathrm dt, \mathrm de))^\top=(\rho_1(\mathrm de)\mathrm dt ,\dots,  \rho_l(\mathrm de)\mathrm dt )^\top$,
where $E=\mathbb{R}^{ l}\backslash\{0\}$ is equipped with its  Borel $\sigma$-field $\mathcal E$, and $\rho_i$ is the given $\sigma$-finite measure on the measurable space $(E, \mathcal{E} )$ and satisfies $\displaystyle\int_{E}\left(1 \wedge|e|^{2}\right) \rho_i(\mathrm de)<\infty$, $i=1,\cdots l$.

  Consider an  infinite horizon mean-field stochastic differential equation (MF-SDE, for short) with jumps as follows,
\begin{equation}\label{state}
\left\{
\begin{aligned}
&\!\!\mathrm dx(s)=\big(A(s) x(s) \!+ \!\bar{A}(s)\mathbb{E}{[x(s)]} \!+\! B(s) u(s)\!+\!\bar{B}(s)\mathbb{E}{[u(s)]}\big) \, \mathrm ds\\
& \!\! \!  +\sum\limits_{i=1}^d\big(C_i(s) x(s)\!+\!\bar{C}_i(s)\mathbb{E}{[x(s)]}\!+\!D_i(s) u(s)\!+\!\bar{D}_i(s)\mathbb{E}{[u(s)]}\big) \, \mathrm dW_i(s) \\
&\!\! \!  +\! \ds\sum\limits_{j=1}^l\! \int_E\!\! \big(M_j(s,e) x(s-)\!+\!\bar{M}_j(s, e)\mathbb{E}{[x(s-)]}\!+\!N_j(s, e) u(s)
\!+\!\bar{N}_j(s, e)\mathbb{E}{[u(s)]}\big)\,  \widetilde{\mu}_j\left(\mathrm ds, \mathrm de\right),\ \! s\geq t, \!\!\!\!\!  \\
&\!\!x(t)=x_{t},
\end{aligned}\right.
\end{equation}
where   $t\in [0,\infty)$ is the initial time,  $x_t\in L_{\mathcal {F}_t}^2(\Omega;\mathbb{R}^n)$ is the initial state  and    $u(\cdot) $ is the control process.
The optimization  object  is introduced as follows,
\begin{equation}\label{cost}
\begin{array}{lll}
&\!\!\!J^{K}\left(t, x_{t} ; u(\cdot)\right) =\dfrac{1}{2} \mathbb{E} \ds\int_{t}^{\infty} e^{2 K s}g\big(s,x(s),\mathbb{E}[ x(s)],u(s),\mathbb{E}[ u(s)]\big)\,
  \mathrm ds,
\end{array}
\end{equation}
where $K\in\mathbb {R}$, and for $(s,x, x',u,u')\in [0,\infty)\times \mathbb{R}^n \times \mathbb{R}^n\times \mathbb{R}^m\times \mathbb{R}^m$,
$$
\begin{array}{lll}
&\!\!\!g(s,x,x',u,u') := \Bigg\langle\!\!\! \left(
                \begin{array}{cc}
               \!\!\! Q(s) & S(s)^\top\!\!\! \\
               \!\!\! S(s) & R(s) \!\!\!\\
                \end{array}
              \right)\!\!\!\left(
                       \begin{array}{c}
                        \!\!\! x\!\!\!\\
                        \!\!\! u\!\!\! \\
                       \end{array}
                     \right),
                     \left(
                       \begin{array}{c}
                      \!\!\!   x\!\!\! \\
                      \!\!\!  u \!\!\!\\
                       \end{array}
                     \right)\!\!\!\Bigg\rangle +\Bigg\langle \!\!\!\left(
                             \begin{array}{cc}
                              \!\!\! \bar{Q}(s) & \bar{S}(s)^\top \!\!\!\\
                              \!\!\! \bar{S}(s) & \bar{R}(s)\!\!\! \\
                             \end{array}
                           \right)
                           \!\!\! \left(
                       \begin{array}{c}
                      \!\!\!  x'\!\!\! \\
                       \!\!\! u'\!\!\! \\
                       \end{array}
                    \right),
                     \left(
                       \begin{array}{c}
                       \!\!\! x' \!\!\!\\
                      \!\!\!  u' \!\!\!\\
                       \end{array}
                     \right)\!\!\!\Bigg\rangle.
\end{array}
$$
 The  details for the coefficients of \eqref{state} and \eqref{cost} are postponed and will be specified  in Section \ref{sec:LQ}.  Now we    formulate the control problem   informally as follows.

 \textbf{Problem (MF-LQ)}
For any initial pair $\left(t, x_{t}\right) \in[0, \infty) \times$ $L_{\mathcal{F}_{t}}^{2}\left(\Omega ; \mathbb{R}^{n}\right)$, find an admissible control $u^{*}(\cdot) \in L^{2,K}_\mathbb{F}(t, \infty;\mathbb{R}^m)$ such that
\begin{equation}\label{equ-MFLQ-5}
J^{K}\left(t, x_{t} ; u^{*}(\cdot)\right)=\inf _{u(\cdot) \in L^{2,K}_\mathbb{F}(t, \infty;\mathbb{R}^m)} J^{K}\left(t, x_{t} ; u(\cdot)\right).
\end{equation}
The    admissible control $u^{*}(\cdot)$ satisfying \eqref{equ-MFLQ-5} is said to be  an  \emph{open-loop optimal  control}  of  Problem (MF-LQ) at   $\left(t, x_{t}\right)$.

The  objective of this paper is to  characterize  the open-loop optimal  controls.
The study  on  the aspect of closed-loop optimal controls  is not involved here,  which will be delayed to our future works.
Note that, infinite horizon, time-varying coefficients, mean-field terms,  Poisson jumps, these  four features   coexist in Problem (MF-LQ). There are indeed some relevant researches involving the above one or two features, referring to
 \cite{AO-2014, DM-2019, GT-2008, HOP-2013,   HLY-2015, LSY-2021, MWY-2021, OV-2017, PZ-2021,  RZ-2000, RM-2022,  Sun-Yong-2018, Tian-Yu-2020, Wei-Yu-2021, Yu-2017}, etc.
 However, these works can not be generalized directly to solve  Problem (MF-LQ). The study of Problem (MF-LQ) turns out to be rather
tricky, subtle and technical, far from the
combination of the four features.
Now, let us display the main contributions of this paper and  make the elaborate explanation.

The first and the most important is  to  formulate Problem (MF-LQ) well.
Here,  we do not   follow the  notion of  stabilizability used in \cite{HLY-2015, Sun-Yong-2018, LSY-2021} for the infinite horizon systems, which is  not clear enough  for  the models with   time-varying coefficients.
 We   put a lot of effort into  studying the global integrability of the solutions to the stochastic linear dynamical systems  involving time-varying coefficients and mean-field terms, just liking \eqref{state}.
 Such a study is also  carried out for  infinite horizon mean-field backward stochastic differential equation (MF-BSDE, for short) with jumps, which will behavior as the adjoint equation of  Problem (MF-LQ).
We introduce a kind of  conditions,   which is  equivalently to the classical  dissipation conditions in dynamical systems, to a certain extent.
 This implies that    the  proposed conditions   are relative classical and acceptable, although   stronger than the stabilizability requirements in \cite{HLY-2015, Sun-Yong-2018, LSY-2021}.
Under  the proposed conditions, we get the needed global wellposedness of MF-SDEs and MF-BSDEs by the subtle derivations and  using the separation technique introduced by Yong \cite{Yong-2013} for mean-field linear models.

 After the well formulation of Problem (MF-LQ),  we  proceed to characterize its open-loop optimal controls.
The variational technique in control theory  is applied here to get the equivalence   between the open-loop optimal control and an optimality condition  (also called a stationary condition).
 During the process,
 the positive definiteness condition  (referring to Condition (PD) in Section 3) on the coefficient matrices of the  cost functional \eqref{cost} is assumed.
Moreover, the positive definiteness condition and the optimality condition together   transform the optimal state and  adjoint equations
 into  the  Hamiltonian system \eqref{Hamil}, which is in fact an infinite horizon fully coupled mean-field forward-backward stochastic differential equations (MF-FBSDEs, for short) with jumps.
Then,  the characterization of  the open-loop optimal controls boils down to the existence and uniqueness of the solutions to such   equations.

To make the study more completely, we  also pay   attention   to establishing the wellposedness of   infinite horizon fully coupled MF-FBSDEs with jumps.  Talking about this,  the similar research  in finite horizon  case is made in \cite{WYY-2019}, where the model    can cover most of finite horizon mean-field LQ  control problems.
In addition,   one of the main approaches of studying fully coupled FBSDEs is
 the method of continuation, which is applicable  no matter   in finite or infinite horizon, such as \cite{HP-1995, Peng-Wu-1999, Peng-Shi, Wei-Yu-2021}, etc.
Therefore, here we try to apply this approach in infinite horizon fully coupled MF-FBSDEs with jumps.
The method itself is not difficult. However,  the construction of the parameterized equations  and the associated a priori estimate need the deeply thought in our framework.  We carry out  the delicate  analysis to conclude  the wellposedness of the  Hamiltonian system \eqref{Hamil} with $S(\cdot)$ and $\bar S(\cdot)$ all being the zero matrices.
For the study of the other circumstance of $S(\cdot)$ and $\bar S(\cdot)$, we do not conduct the  method of continuation, but resort to a kind of linear transformation.  According to the distinguishing study, we find that  $S(\cdot)$ and $\bar S(\cdot)$ being zero or not will aspect the existing space of the optimal controls, which will not happen in finite horizon.

  Besides the above mentioned contributions, another special  feature   is the appearance of constant $K$, which  is in fact inspired by our  previous work \cite{Wei-Yu-2021}. On one hand, the study here  with $K\in\mathbb{R}$ can be regarded as a generalization of   \cite{Wei-Yu-2021} to the problems concerning mean-field and jumps. Meanwhile,
  it can cover a lot of the existing models by setting  $K=0.$
 On the other hand, the introduction of $K\in\mathbb{R}$  is  friendly with various models, regardless of  the degree of dissipation conditions. We can always identify the  parameter $K$  according to the intrinsic properties of the coefficient matrices, and formulate the control problems very well.
As indicated in  \cite{Wei-Yu-2021}, when the eigenvalues of $A(\cdot )$ are negative enough, the state
process $x(\cdot )$ itself will admit an exponential decay, so that the parameter $K$ may take
a positive real number.
By the way, except \cite{Yu-2017}, there is rarely infinite horizon  LQ control  problem involving the Poisson random measures, which owns the practical economic background.
Here we incorporate  the control problem with the multi-dimensional Poisson random measures, which is quite comprehensive.

The rest of this paper is structured as follows: Section 2 presents the necessary notations and explores the wellpposedness of the infinite horizon MF-SDEs and MF-BSDEs with jumps.  Section 3 is about the infinite horizon mean-field    LQ problem with jumps.
 By systematically establishing the wellposedness of infinite horizon fully coupled MF-FBSDEs with jumps, we give a more straightforward and complete characterization of the open-loop optimal controls.
An example is presented to illustrate the different behaviors when  the cross terms $S(\cdot)$, $\bar S(\cdot)$ are zero or not.

\section{Preliminaries}

We shall use $\langle\cdot, \cdot\rangle$ and $|\cdot|$ to represent the Euclidean inner product and the Euclidean norm, respectively.  For the Euclidean space of matrices, we also use the following  operator norm,
$$ \left\Vert \mathbf{\Lambda} \right\Vert:=\sup _{0 \neq x \in \mathbb{R}^{n}} \frac{|\mathbf{\Lambda} x|}{|x|} \quad \text { for any } \mathbf{\Lambda} \in \mathbb{R}^{m \times n}.
 $$

 In the context, $\mathbb{S}^n$ is the set of $n \times n$ real symmetric matrices.
When a matrix $\mathbf{\Lambda} \in \dbS^{n}$ is positive semidefinite (resp., positive definite, negative semidefinite, negative definite), we denote $\mathbf{\Lambda} \geq 0$ (resp., $>0,$ $\leq 0,$ $<0$).
Moreover, for a matrix-valued function $\mathbf{\Lambda}(\cdot):[0, \infty) \rightarrow \dbS^{n}$, if $\mathbf{\Lambda}(s) \geq 0$ (resp., $ >0, $ $\leq 0, $ $<0)$ for almost all $s \in[0, \infty) $, then we denote $\mathbf{\Lambda}(\cdot) \geq 0$ (resp., $>0, $ $\leq 0$, $<0)$.
Furthermore, if there exists a constant $\d>0$ such that $\mathbf{\Lambda}(\cdot)-\d I \geq 0$ (resp., $\mathbf{\Lambda}(\cdot)+\d I \leq 0$),
we denote $\mathbf{\Lambda}(\cdot) \gg 0$ (resp., $\mathbf{\Lambda}(\cdot) \ll 0$).\par

 Now we introduce  some spaces of random variables or stochastic processes as follows. For any Euclidean space $\mathbb{H}$, constant $K \in \mathbb{R}$ and $0\leq t\leq s<\infty$,
\begin{itemize}
 \item $\!\!L_{\mathcal{F}_{s}}^{2}\left(\Omega ; \mathbb{H}\right)$ is the set of all the  $\mathcal{F}_{s}$-measurable  $\xi: \Omega\to\mathbb{H}$  satisfying $ \mathbb{E}|\xi|^{2}<\infty$;


%
 \item $\!\!L^{\infty}\left(t, \infty ; \mathbb{H}\right)$ is the set of the Lebesgue measurable $\varphi: [t, \infty)\to\mathbb{H} $ satisfying  $\sup\limits_{s \in[t, \infty)}|\varphi(s)|<\infty $;

 \item $\!\!\displaystyle
 L_{\mathbb{F}}^{2, K}\left(t, \infty; \mathbb{H}\right) $ is the set of $\mathbb{F}$-progressively measurable $\varphi: \Omega\times[t, \infty)\to
 \mathbb{H} $ such that
 $$\mathbb{E} \int_{t}^{\infty}\left|\varphi(s) \mathrm e^{K s}\right|^{2} \,\mathrm ds <\infty ;$$

 \item $\!\!\ds L^{2}_\rho (E; \mathbb {R}^{n\times l})$ is the set of all the $\cE$-measurable $\psi=(\psi_1,\dots,\psi_l): E\to\mathbb {R}^{n\times l}  $ such that
     $$ \|\psi(\cd)\|_\rho:=\Big(\int_E   \tr[\psi(e)\varrho(\mathrm de) \psi(e)^\top]\Big)^\frac12 = \(\sum_{j=1}^l \int_E \left|\psi_j(e)\right|^{2}\, \rho_j(\mathrm de)\)^\frac12<\infty,$$
where $\varrho(\cdot):=\diag(\rho_1(\cdot),\dots,\rho_l(\cdot))$.
For any $\psi(\cdot),$    $\bar\psi(\cdot)\in L^{2}_\rho (E; \mathbb {R}^{n\times l})$, the associated  inner product is introduced  as,
 $$\displaystyle\langle \psi,\bar \psi\rangle_\rho:=  \int_E \tr[\psi(e)\varrho(\mathrm de)\bar\psi(e)^\top] =\sum_{j=1}^l \int_E \bar\psi_j(e)^\top  \psi_j(e)    \, \rho_j(\mathrm de);$$

 \item $\!\! \displaystyle \cK_{\mathbb{F}}^{2, K}(t, \infty ; \mathbb {R}^{n\times l}) $ is the set of all the $\mathcal{P}_t \otimes \cE $-measurable $k=(k_1,\dots,k_l): \Omega\times [t, \infty)\times E\to\mathbb {R}^{n\times l}$ such that
 $$
  \mathbb{E} \int_{t}^{\infty} \|k(s, \cdot)\mathrm e^{K s}\|_\rho ^{2} \,\mathrm ds <\i,$$
 where $\mathcal{P}_t$ is the $\sigma$-algebra generated by the $\dbF$-progressively measurable processes on $ \Omega\times [t, \infty)$.

\end{itemize}
%
%
%
%
For any $K_{1}<K_{2}$, it is obvious that $L_{\mathbb{F}}^{2, K_2}\left(t, \infty ; \mathbb{H} \right) \subset L_{\mathbb{F}}^{2, K_{1}}\left(t, \infty ; \mathbb{H} \right)$ and $\cK_{\mathbb{F}}^{2, K_{2}}\left(t, \infty ; \mathbb{H}\right) \subset \cK_{\mathbb{F}}^{2, K_{1}}\left(t, \infty ; \mathbb{H}\right)$. For  convenience of the later study, we use the following abbreviated notations,
$$\ba{ll}
\ns\ds \mathscr{R}:=  \mathbb{R}^{n} \times \mathbb{R}^{n} \times \mathbb{R}^{n  d} \times L_\rho^{2} (E; \mathbb{R}^{n l} ), \\
\ns\ds \mathcal{L}_{\mathbb{F}}^{2, K}(t, \infty):=
L_{\mathbb{F}}^{2, K}(t, \infty ; \mathbb{R}^{n})
\times L_{\mathbb{F}}^{2, K} (t, \infty ; \mathbb{R}^{n})
\times L_{\mathbb{F}}^{2, K}(t, \infty ; \mathbb{R}^{n  d})
\times \cK_{\mathbb{F}}^{2, K}(t, \infty ; \mathbb{R}^{n l  }).
\ea$$
And the inner product $\langle \cdot,\cdot\rangle$ and the norm $|\cdot|$ of $\mathscr{R}$ are introduced as follows,
for any $\th =( x, y, z, k(\cdot)), $ $\th' =( x', y', z', k'(\cdot))\in\mathscr{R}$,
$$\ba{ll}
\ns\ds \langle \theta, \th'\rangle := \langle x, x'\rangle+\langle y, y'\rangle+\langle z, z'\rangle+\langle k(\cdot), k'(\cdot)\rangle _\rho, \q |\th|:=\sqrt{\langle \theta, \theta\rangle} .
\ea $$
Naturally, for any $\th(\cdot):=( x(\cdot), y(\cdot), z(\cdot), k(\cdot,\cdot))\in\cL_{\mathbb{F}}^{2, K}(t, \infty)$, its norm is introduced as
$$\ba{ll}
\displaystyle \|\theta(\cdot)\|_{\cL_{\mathbb{F}}^{2, K}} :=\left\{\mathbb{E}     \int_{t}^{\infty}\Big(|x(s)\mathrm e^{K s}|^{2}
+ |y(s)\mathrm e^{K s}|^{2}
+ |z(s)\mathrm e^{K s}|^{2}
+\|k(s, \cd)\mathrm e^{K s} \|_\rho^{2}\Big)\mathrm ds   \right\}^{\frac{1}{2}}.
\ea $$
%
%

 \smallskip

In the context, for any   involved   random variables $\zeta$, we shall always use $$\zeta^{(1)}:=\zeta-\mathbb{E}[\zeta],\quad \zeta^{(2)}:=\mathbb{E}[\zeta].$$
 Obviously,
$
\mathbb{E}|\zeta|^{2}= \mathbb{E} |\zeta^{(1)}|^{2} +|\zeta^{(2)}|^{2}.
$
Moreover,  we  set
  $\Gamma^1(\cdot):=\Gamma(\cdot) $,   $\Gamma^2(\cdot):=\Gamma (\cdot)+\bar \Gamma(\cdot)$ for
all the matrix-valued functions  $\Gamma(\cdot)$, $\bar \Gamma(\cdot)$   appearing in the following sections.

\subsection{Infinite horizon   mean-field  SDEs with jumps}

For any initial pair $(t, x_t)\in [0, \infty)\times L_{\cF_t}^{2}(\Omega;\mathbb{R}^n)$,  we consider the following  infinite horizon  mean-field SDE  with jumps,
\begin{equation}\label{equ-SDE-1}
\left\{
\begin{aligned}
&\!\!  \mathrm dx(s)\! =\!\big(A(s) x(s) + \bar{A}(s)\mathbb{E}{[x(s)]} + b(s)\big)  \mathrm ds
+\sum\limits_{i=1}^d\big(C_i(s) x(s)+\bar{C}_i(s)\mathbb{E}{[x(s)]}+\sigma_i(s)\big)  \mathrm dW_i(s)\! \\
&\hskip1.2cm
+\ds\sum\limits_{j=1}^l\int_E \big(M_j(s, e) x(s-)+\bar{M}_j(s, e)\mathbb{E}{[x(s-)]}+\gamma_j(s, e)\big) \widetilde{\mu}_j(\mathrm ds, \mathrm de),\qquad  s\geq t,\\
&\!\!   x(t) \!= \!x_t,
\end{aligned}
\right.
\end{equation}
where, for $i=1,\cdots,d,$ $ j=1,\cdots, l$,  the coefficients $A(\cdot),$ $\bar{A}(\cdot),$ $C_i(\cdot),$ $\bar{C}_i(\cdot),$ $M_j(\cdot, \cdot),$ $\bar{M}_j(\cdot, \cdot)$  are the matrix-valued  deterministic  functions,   the nonhomogeneous terms   $b(\cdot), $ $\sigma_i(\cdot)$, $\gamma_j(\cdot, \cdot)$ are the vector-valued stochastic processes.
 To simplified the notations,    set
  $$\begin{aligned}
& \mathbf{C}(\cdot)=(C_1(\cdot)^\top,\dots,C_d(\cdot)^\top)^\top,\qquad  \bar {\mathbf{C}}(\cdot)=(\bar C_1(\cdot)^\top,\dots,\bar C_d(\cdot)^\top)^\top,\\
 &\mathbf{M}(\cdot,\cdot)=(M_1(\cdot,\cdot)^\top,\dots,M_l(\cdot,\cdot)^\top)^\top,\qquad \bar {\mathbf{M}}(\cdot,\cdot)=(\bar M_1(\cdot,\cdot)^\top,\dots,\bar M_l(\cdot,\cdot)^\top)^\top,\\
 &\sigma(\cdot)=(\sigma_1(\cdot)^\top,\dots,\sigma_d(\cdot)^\top)^\top,\qquad \gamma(\cdot,\cdot)=(\gamma_1(\cdot,\cdot)^\top,\dots,\gamma_l(\cdot,\cdot)^\top)^\top.
 \end{aligned}$$
%
%
Next, we assume   the following conditions on  the above functions.

\smallskip
\noindent\textbf{Assumption (H$_{1}$)}
$A(\cdot), \bar{A}(\cdot)\in L^{\infty}\left(0, \infty; \mathbb{R}^{n \times n}\right)$,
$\mathbf{C}(\cdot), \bar {\mathbf{C}}(\cdot)\in L^{\infty}\left(0, \infty ; \mathbb{R}^{(nd)   \times n}\right),$
$\mathbf{M}(\cdot, \cdot), \bar {\mathbf{M}}(\cdot, \cdot)\in  L^{\infty}\left(0, \infty ; L^{2}_\rho \left(E;  \mathbb{R}^{(nl) \times n}\right)\right)$.
Moreover,   $b(\cdot) \in L_{\mathbb{F}}^{2,  K_1}\left(0, \infty ; \mathbb{R}^{n}\right)$, $\sigma(\cdot) \in L_{\mathbb{F}}^{2, K_1}\left(0, \infty ; \mathbb{R}^{n d }\right)$, $\gamma(\cdot, \cdot) \in \cK_{\mathbb{F}}^{2,  K_1}\left(0, \infty ; \mathbb{R}^{nl}\right)$ for some constant  $ K_1\in\mathbb{R}$.
\ms
It is classical that,
under {\bf{(H$_{1}$)}}, for any $(t, x_t)\in [0, \infty)\times L_{\cF_t}^{2}(\Omega;\mathbb{R}^n)$, \eqref{equ-SDE-1} admits  the unique $\dbF$-adapted solution $x(\cd)$.
Moreover,   for any $T<\infty$, $
\mathbb{E}\big[\sup\limits_{s\in[t,T]}|x(s)|^2\big]<\infty.
$
Similarly to  the case without mean-field terms, we have the following result  for mean-field SDE
 \eqref{equ-SDE-1}. The details may  be referred  to \cite{Yu-2017}.

\begin{remark}\label{Re-SDE-1}\sl
Let {\bf{(H$_{1}$)}} hold. If the solution $x(\cdot)$ to MF-SDE with jumps \eqref{equ-SDE-1} belongs to $L_{\mathbb{F}}^{2, K}\left(t, \infty ; \mathbb{R}^{n}\right)$, then
$
\displaystyle\lim_{s \rightarrow \i} \mathbb{E}\big[ |x(s) \mathrm e^{ K s} |^{2}\big]=0.
$
\end{remark}
Let us make the following abbreviations,
\begin{equation}\label{kappa}\left\{
\begin{aligned}
&\!\!\kappa_1:=-\frac{1}{2}\mathop{\sup}\limits_{s \in[t, \infty)}\lambda_{\max }\( \!A^{2}(s) +
 A^{2}(s) ^{\top} \!\),\\
&\!\!\kappa_2:=-\frac{1}{2}\mathop{\sup}\limits_{s \in[t, \infty)}\lambda_{\max }\( \! A^{1}(s)+ A^{1}(s)^{\top}+ \mathbf{C}^{1}(s)^{\top}\mathbf{C}^{1}(s)
+\int_E \mathbf{M}^{1}(s,e)^\top\varrho(\mathrm de)\mathbf{M}^{1}(s,e)\!\),\\
&\!\! \kappa :=\min\{\kappa_1,\kappa_2 \},
\end{aligned}\right.
\end{equation}
where $ \lambda_{\max}(\cdot)$
represents the largest eigenvalue of the matrix.
And note that, for $i=1,2$, $\mathbf{M}^{i}(\cdot, \cdot)=(   {M}^{i}_1(\cdot, \cdot)^\top,\dots, {M}^{i}_l(\cdot, \cdot)^\top  )^\top$ with  $  {M}^{1}_j  (\cdot, \cdot)=M_j(\cdot, \cdot)$,   $   {M}^{2}_j  (\cdot, \cdot)=M_j(\cdot, \cdot)+\bar M_j(\cdot, \cdot)$, $j=1,\cdots,l$.

\smallskip

Now we present the global integrability  of the solution to mean-field SDE \eqref{equ-SDE-1}.
\begin{lemma}\label{Le-SDE-1}\sl
Assume that {\bf{(H$_{1}$)}} holds.
Then the solution $x(\cdot)$ to MF-SDE with jumps \eqref{equ-SDE-1} belongs to $L_{\mathbb{F}}^{2, K}\left(t, \i; \mathbb{R}^{n}\right)$ with
$ K \leq K_1$ and $K< \kappa$.
And for any $ \varepsilon \in(0, \frac{-2K+2\kappa }{3})$, we have
\begin{equation}\label{equ-SDE-3}
\begin{aligned}
&(- 2K+ 2\kappa - 3\varepsilon  ) \mathbb{E}\!\int_t^\infty\! |x(r)\mathrm e^{ Kr}|^{2}  \, \mathrm  dr\leq L_{\varepsilon,1} \mathbb{E} |x_t \mathrm e^{ K t}|^{2}  \\
&\qquad+  \mathbb{E}\!\int_t^\infty\! \[\frac {L_{\varepsilon,1}}{\varepsilon}|b(r)|^{2}+\(\!2+\frac{\|\mathbf{C}^{1}(r)\|^{2}}{\varepsilon}\!\)|\sigma(r)|^{2}
+\(\! 2+\frac{\|\mathbf{M}^{1}(r, \cdot)\|_\rho^{2}}{\varepsilon}\!\)\|\gamma(r, \cdot)\|_\rho^{2} \]\mathrm e^{2 K r} \mathrm  dr,\!
\end{aligned}
\end{equation}
where $L_{\varepsilon,1}:= 1+\frac{ 2\sup\limits_{r\in[0,\i)}\big(\|\mathbf{C}^{2}(r) \|^{2}+ \|\mathbf{M}^{2}(r, \cdot) \|_\rho^{2}\big)}{-2K + 2\kappa_1  -\varepsilon}$.
\medskip
Further, let  $\bar{x}(\cdot) \in L_{\mathbb{F}}^{2,  K}\left(t, \infty ; \mathbb{R}^{n}\right)$   be the solution to
MF-SDE with jumps \eqref{equ-SDE-1} with another initial state $\bar{x}_{t} \in L_{\mathcal{F}_t}^2(\Omega;\mathbb{R}^n)$ and   nonhomogeneous terms $ \bar{b}(\cdot), \bar{\sigma}(\cdot), \bar{\gamma}(\cdot)  $ satisfying {\bf{(H$_{1}$)}}.
Then, for any $ \varepsilon \in(0, \frac{-2K+2\kappa }{3})$, we have
\begin{equation}\label{equ-SDE-4}
\begin{aligned}
&(- 2K+ 2\kappa  - 3\varepsilon ) \mathbb{E}\!\int_t^\infty\! |(x(r)-\bar{x}(r))e^{Kr}|^{2}  \, \mathrm  dr\\
&\leq L_{\varepsilon,1}\mathbb{E} |(x_t- \bar{x}_t)e^{K t}|^{2}
+  \mathbb{E}\!\int_t^\infty \!\[\frac {L_{\varepsilon,1}}{\varepsilon}|b(r)-\bar{b}(r)|^{2}
+\(\! 2+\frac{\|\mathbf{C}^{1}(r)\|^{2}}{\varepsilon}\!\)|\sigma(r)-\bar{\sigma}(r)|^{2} \\
&\hskip5.2cm
+\(\!2+\frac{\|\mathbf{M}^{1}(r, \cdot)\|_\rho^{2}}{\varepsilon}\!\) \|\gamma(r, \cdot)-\bar{\gamma}(r, \cdot)\|_\rho^{2} \] \mathrm e^{2 K r} \,\mathrm  dr.
\end{aligned}
\end{equation}
\end{lemma}
\ms
\begin{proof}
Firstly, let us make an observation. By \eqref{equ-SDE-1}, we have $x(\cdot)=x^{(1)}(\cdot)  + x^{(2)}(\cdot) $ with%
$$
\left\{
\begin{aligned}
&\!\! \mathrm d x^{(1)}(s) =\big(A^{1}(s)x^{(1)}(s)  + b^{(1)}(s)\big) \, \mathrm ds
+\sum\limits_{i=1}^d\big(C_i^{1}(s) x^{(1)}(s)+ C^{2}_i(s) x^{(2)}(s) +\sigma_i(s)\big)  \mathrm dW_i(s) \\
&\!\! \hskip1.5cm
+\ds\sum\limits_{j=1}^l\int_E \big(M_j^{1}(s, e) x^{(1)}(s-)+ {M}^{2}_j(s, e)  x^{(2)}(s-)  +\gamma_j(s, e)\big) \widetilde{\mu}_j(\mathrm ds, \mathrm de), \q s\geq t,\\
&\!\! x^{(1)}(t) = x_t-\mathbb{E}{[x_t]},
\end{aligned}
\right.
$$
and
$$
\left\{
\begin{aligned}
&\!\! \mathrm dx^{(2)}(s) =\big( A^{2}(s) x^{(2)}(s) + b^{(2)}(s)\big) \, \mathrm ds , \quad s\geq t, \\
&\!\!  x^{(2)}(t)  = \mathbb{E}{[x_t]}.
\end{aligned}
\right.
$$

For any  $\varepsilon>0$,  applying It\^o's formula to $\big| x^{(1)}(\cdot) \mathrm e^{ K \cdot}\big|^{2}$, we get
$$
\begin{aligned}
& \mathbb{E}|x^{(1)}(T) \mathrm e^{K T}|^2- \mathbb{E}|(x_t-\mathbb{E}{[x_t]}) \mathrm  e^{K t}|^2\\
&=\mathbb{E}\int_t^T\[\big\langle \big(A^{1}(r)+ A^{1}(r)^{\top}+2KI\big)x^{(1)}(r), x^{(1)}(r)\big\rangle
\big\rangle  +  \big|\mathbf{C}^1(r) x^1(r)+\mathbf{C}^2(r)  x^2(r)+\sigma(r)\big|^2 \\
&\hskip0.6cm
+\|\mathbf{M}^{1}(r,\cd) x^{(1)}(r)+\mathbf{M}^{2}(r,\cd)  x^{(2)}(r)+\gamma(r, \cd)\|_\rho^2 + 2\big\langle x^{(1)}(r), b^{(1)}(r) \big\rangle\] \mathrm e^{2 K r}\, \mathrm  dr \\
&=\mathbb{E}\int_t^T\[\big\langle\big(A^{1}(r)+ A^{1}(r)^{\top}+\mathbf{C}^{1}(r)^{\top}\mathbf{C}^{1}(r)+\int_E \mathbf{M}^{1}(s,e)^\top\varrho(\mathrm de)\mathbf{M}^{1}(s,e)+2KI\big) x^{(1)}(r), x^{(1)}(r)\big\rangle \\
&\hskip0.6cm
+\big|\mathbf{C}^{2}(r)  x^{(2)}(r)+\sigma(r)\big|^2
+2\big\langle \mathbf{C}^{1}(r)x^{(1)}(r), \sigma(r)\big\rangle+ 2\big\langle x^{(1)}(r), b^{(1)}(r)\big\rangle\\
&\hskip0.6cm
+ \|\mathbf{M}^{2}(r,\cd)  x^{(2)}(r)+\gamma(r, \cd) \|_\rho^2
+2\langle \mathbf{M}^{1}(r,\cd)x^{(1)}(r), \gamma(r, \cd)\rangle_\rho \] \mathrm e^{2 K r}\, \mathrm dr \\
%
&\leq \mathbb{E}\int_t^T \[ \big(2K-2\kappa_2+3\varepsilon \big)|x^{(1)}(r)|^{2}
+2\big(\|\mathbf{C}^{2}(r) \|^{2}+   \|\mathbf{M}^{2}(r,\cdot) \| _\rho^{2} \big)|x^{(2)}(r)|^{2}\\
&\hskip0.6cm+\frac{1}{\varepsilon}|b^{(1)}(r) |^{2}
+\Big(2+\frac{\|\mathbf{C}^{1}(r)\|^{2}}{\varepsilon}\Big)|\sigma(r)|^{2}
+\Big(2+\frac { \|\mathbf{M}^{1}(r,\cdot) \| ^{2}_\rho}{\varepsilon} \Big) \|\gamma(r, \cdot) \|_\rho^{2} \] \mathrm e^{2 K r} \,\mathrm  dr.
\end{aligned}
$$
Therefore, when $  K \leq  K_1$,  we get
\begin{equation}\label{equ-SDE-5}
\begin{aligned}
&(- 2K+2\kappa_2  -3\varepsilon )\mathbb{E}\int_t^T  |x^{(1)}(r)  e^{ K r}|^{2}   \, \mathrm  dr\\
&\leq \mathbb{E}|(x_t-\mathbb{E}{[x_t]}) e^{K t}|^2+2\sup\limits_{r\in[t,\i)}\big(\|\mathbf{C}^{2}(r) \|^{2}+ \|\mathbf{M}^{2}(r, \cdot) \|_\rho^{2}\big)\mathbb{E}\int_t^T
 |x^{(2)}(r)|^{2} \,\mathrm  dr\\
& +\mathbb{E}\int_t^T\[\frac{1}{\varepsilon}|b^{(1)}(r)|^{2}
\!+\!\Big(2+\frac{\|\mathbf{C}^{1}(r)\|^{2}}{\varepsilon}\Big)|\sigma(r)|^{2}
\!+\!\Big(2+\frac{\|\mathbf{M}^{1}(r, \cdot)\|_\rho^{2}}{\varepsilon}\Big) \|\gamma(r, \cdot) \|_\rho^{2} \] \mathrm e^{2 K r} \,\mathrm  dr.\!\!
\end{aligned}
\end{equation}

%
Similarly, for any  $\varepsilon>0$,    we have
$$\begin{aligned}
& | x^{(2)}(T)  \mathrm e^{ K T}  |^{2}  \\
&=|\mathbb{E}{[x_t \mathrm e^{ K t}]} |^{2}
+ \int_t^T\[\big\langle \big(A^{2}(r)+A^{2}(r)^{\top}\!+2KI\big) x^{(2)}(r) ,  x^{(2)}(r)  \big\rangle
+ 2\big\langle  x^{(2)}(r), b^{(2)}(r) \big\rangle \]e^{2 K r}\,\mathrm  dr\\
&\leq |\mathbb{E}{[x_t  \mathrm e^{ K t}]}  |^{2}
+(2K-2\kappa_1+\varepsilon)
\int_t^T |x^{(2)}(r) \mathrm  e^{K r}   |^{2}\,\mathrm  dr+ \frac{1}{\varepsilon}\int_t^T |b^{(2)}(r)  \mathrm e^{K r}|^{2}\, \mathrm  dr.
\end{aligned}$$
 Then, when $ K\leq K_1$,   we get
\begin{equation}\label{equ-SDE-6}\begin{aligned}
&  (-2K+2\kappa_1  -\varepsilon)
\int_t^T | x^{(2)}(r) \mathrm e^{K r} |^{2}\,\mathrm  dr \leq |\mathbb{E}{[x_t \mathrm  e^{ K t}]}  |^{2}
+  \frac{1}{\varepsilon}\int_t^T |b^{(2)}(r) \mathrm  e^{K r} |^{2}\, \mathrm  dr.
\end{aligned}\end{equation}

Further, when  $  \varepsilon \in(0, \frac{-2K+2\kappa }{3})$, the combination of  \eqref{equ-SDE-5} and \eqref{equ-SDE-6} yields
$$
\begin{aligned}
&(- 2K+ 2\kappa - 3\varepsilon  ) \mathbb{E}\int_t^T|x(r)\mathrm e^{ K r}|^{2}\mathrm  \, \mathrm  dr\\
&\leq  (- 2K+ 2\kappa_1- \varepsilon  ) \int_t^T |x^{(2)}(r) \mathrm e^{ K r} |^{2}  \, \mathrm  dr
+  (- 2K+2 \kappa_2- 3\varepsilon  )\mathbb{E}\int_t^T |x^{(1)}(r) \mathrm e^{ K r}|^{2}   \, \mathrm  dr\\
&\leq    \mathbb{E} |x_t  \mathrm e^{ K t}  |^{2}
+  \frac{1}{\varepsilon}\int_t^T\mathbb{E} | b(r)  \mathrm e^{ K r} |^{2}\, \mathrm  dr
 + 2\sup\limits_{r\in[0,\i)}\big(\|\mathbf{C}^{2}(r) \|^{2}+ \|\mathbf{M}^{2}(r, \cdot) \|_\rho^{2}\big)\mathbb{E}\int_t^T
 |x^{(2)}(r) \mathrm e^{ K r}|^{2} \,\mathrm  dr\\
&\quad +\mathbb{E}\int_t^T\[ \Big(2+\frac{\|\mathbf{C}^{1}(r)\|^{2}}{\varepsilon}\Big)|\sigma(r)|^{2}
+\Big(2+\frac{\|\mathbf{M}^{1}(r, \cdot)\|_\rho^{2}}{\varepsilon}\Big) \|\gamma(r, \cdot) \|_\rho^{2} \] \mathrm e^{2 K r} \,\mathrm  dr
\\
&\leq L_{\varepsilon,1}\mathbb{E} |x_t \mathrm e^{ K t}|^{2}   \\
&\quad+  \mathbb{E}\int_t^T \(\frac {L_{\varepsilon,1}}{\varepsilon}|b(r)|^{2}+\(2+\frac{\|\mathbf{C}^{1}(r)\|^{2}}{\varepsilon}\)|\sigma(r)|^{2}
+\(2+\frac{\|\mathbf{M}^{1}(r, \cdot)\|_\rho^{2}}{\varepsilon}\)\|\gamma(s, \cdot)\|_\rho^{2} \) \mathrm e^{2 Kr} \,\mathrm  dr,
\end{aligned}
$$
where  $L_{\varepsilon,1}:= 1+\frac{ 2\sup\limits_{r\in[0,\i)}\big(\|\mathbf{C}^{2}(r) \|^{2}+ \|\mathbf{M}^{2}(r, \cdot) \|_\rho^{2}\big)}{-2K + 2\kappa_1  -\varepsilon}$. Letting $T\to\i$, we get the desired  \eqref{equ-SDE-3}.
And the density of real numbers brings us $x(\cdot)\in L_{\mathbb{F}}^{2, K}\left(t, \i; \mathbb{R}^{n}\right)$ with
$ K \leq K_1$ and $K< \kappa$.

\smallskip

Finally, if setting $\hat{x}(\cdot)=x(\cdot)-\bar{x}(\cdot)$,
and applying the same approach  to $|\mathbb{E}{[\hat{x}(\cdot)]}\mathrm e^{ K \cdot}|^{2}$ and
$|(\hat{x}(\cdot)-\mathbb{E}{[\hat{x}(\cdot)]})\mathrm e^{ K \cdot}|^{2}$ on the interval $[t, T]$,
we can get the estimate \eqref{equ-SDE-4}.

\end{proof}

\begin{remark}\label{Re-2-3} \sl Let us  consider a special  version of \eqref{equ-SDE-1} as follows,
\begin{equation}\label{equ-SDE-c-b}
\left\{
\begin{aligned}
&\!\!  \mathrm dx(s)\! =\!\big(A  x(s) + \bar{A} \mathbb{E}{[x(s)]} + b(s)\big)  \mathrm ds
+ \sum\limits_{i=1}^d\big(C_i(s) x(s)+\bar{C}_i(s)\mathbb{E}{[x(s)]}+\sigma_i(s)\big)  \mathrm dW_i(s),\   s\geq t, \!\!\!\!\!\\
&\!\!   x(t) \!= \!x_t.
\end{aligned}
\right.
\end{equation}
Note that, the   coefficients are constant matrices and the jumps disappear.
There are  some existing  results about the global integrability of \eqref{equ-SDE-c-b} with $b(\cdot)\in L_\mathbb{F}^2(t,\infty;\mathbb{R}^n),$  $\sigma(\cdot)\in L_\mathbb{F}^2(t,\infty;\mathbb{R}^{nd})$, referring to \cite{HLY-2015, LSY-2021}.
The  notion of  stabilizability was used therein.
%
%
%
Our Lemma \ref{Le-SDE-1}  also works in  this case.
 $b(\cdot)\in L_\mathbb{F}^{2}(t,\infty;\mathbb{R}^n)$, $\sigma(\cdot)\in L_\mathbb{F}^{2}(t,\infty;\mathbb{R}^n)$  and
\begin{equation} \label{A+A<0}
  A+\bar A +
( A +\bar A)^\top  <0, \quad  A+ A^{\top}+ \mathbf{C}^{\top}\mathbf{C}<0,
\end{equation}
can guarantee  $x(\cdot)\in L_\mathbb{F}^2(t,\infty;\mathbb{R}^n)$.

However,  when   $b(\cdot)\equiv b\in\dbR^n$, $\sigma (\cdot)\equiv\sigma \in \dbR^{nd}$ ($b$, $\sigma$ being the nonzero constant matrices), the   results in \cite{HLY-2015, LSY-2021}  are invalid  due to $b(\cdot)  \notin   L_\mathbb{F}^2(0,\infty;\mathbb{R}^n),$ $  \sigma(\cdot)\notin L_\mathbb{F}^2(0,\infty;\mathbb{R}^{nd}).$
But for any  $K_1< 0$, we know  $b \in L_\mathbb{F}^{2,K_1}(0,\infty;\mathbb{R}^n)$, $\sigma \in L_\mathbb{F}^{2,K_1}(0,\infty;\mathbb{R}^n)$. Combined with \eqref{A+A<0}, Lemma \ref{Le-SDE-1} can be applied to bring us  $x(\cdot)\in L_\mathbb{F}^{2,K}(t,\infty;\mathbb{R}^n)$ with $K\leq K_1$.

%
\end{remark}


%

 %
%
\subsection{Infinite horizon  mean-field BSDEs with jumps} 
\label{sub:subsection_name}
%
For the same matrix-valued functions $A(\cdot)$, $\bar A(\cdot), $ $\mathbf{C}(\cdot)$, $\bar {\mathbf{C}}(\cdot), $ $\mathbf{M}(\cdot,\cdot)$, $\bar {\mathbf{M}}(\cdot,\cdot)$ in {Assumption {\bf (H$_{1}$)}}, we consider   the following linear infinite horizon    MF-BSDE with jumps,
\begin{equation}\label{equ-BSDE-1}
\begin{aligned}
&\!\mathrm dy(s)=-\Big\{ A_K(s)^{\top}y(s)+\bar{A}(s)^{\top}\mathbb{E}{[y(s)]}+\sum\limits_{i=1}^dC_i(s)^{\top}z_i(s)
+\sum\limits_{i=1}^d\bar{C}_i(s)^{\top}\mathbb{E}{[z_i(s)]}
 \\
&\hskip1.4cm+\sum\limits_{j=1}^l\int_E  M_j(s, e)^{\top}k_j(s, e)  \rho_j(\mathrm de)
+\sum\limits_{j=1}^l\int_E  \bar{M}_j(s, e)^{\top} \mathbb{E}{[k_j(s, e)]} \rho_j(\mathrm de)+ f(s) \Big\} \, \mathrm ds \\
&\hskip1.4cm+\sum\limits_{i=1}^dz_i(s) \, \mathrm dW_i(s)
+\sum\limits_{j=1}^l\int_E k_j(s, e)   \widetilde{\mu}_j(\mathrm ds, \mathrm de), \quad   s \geq t,
\end{aligned}
\end{equation}
where $A_K(\cdot):=A(\cdot)+2KI$,  $I$ is the $(n \times n)$-identity matrix. For convenience, we also denote $z=(z_1^\top,\dots,z_d^\top)^\top $ and $k=(k_1^\top,\dots,k_l^\top)^\top $.

%
The  triple of processes $\vartheta(\cdot):=(y(\cdot), z(\cdot), k(\cdot, \cdot)) \in L_{\mathbb{F}}^{2, K}\left(t, \infty ; \mathbb{R}^{n}\right) \times L_{\mathbb{F}}^{2, K}\left(t, \infty ; \mathbb{R}^{n d}\right) \times \cK_{\mathbb{F}}^{2, K}(t, \infty ;$ $\left.\mathbb{R}^{nl}\right)$  is said to be the solution
to   MF-BSDE with jumps \eqref{equ-BSDE-1} if and only if, for any $T \geq t$,
\begin{equation}\label{equ-BSDE-2}
\begin{aligned}
& y(s)\!=\!y(T)\!+\int_s^T\Big\{A_K(r)^{\top}y(r)\!+\!\bar{A}(r)^{\top}\mathbb{E}{[y(r)]}\!+\!\mathbf{C}(r)^{\top}z(r)\!
+\!\bar{\mathbf{C}}(r)^{\top}\mathbb{E}{[z(r)]}
 \\
&\hskip1.2cm\!+\! \int_E \big( \mathbf{M}(r, e)^{\top} \varrho(\mathrm de)k(r, e)
+  \bar{\mathbf{M}}(r, e)^{\top}\varrho(\mathrm de) \mathbb{E}{[k(r, e)]} \big)\!+\! f(r) \Big\} \, \mathrm dr \!\\
&\hskip1.2cm
\!-\! \sum\limits_{i=1}^d\int_{s}^{T} z_i(r) \, \mathrm dW_i(r)\!-\! \sum\limits_{j=1}^l\int_{s}^{T}\!\int_E k_j(r, e)   \widetilde{\mu}_j(\mathrm dr, \mathrm de),\q s\in[t,T].
\end{aligned}
\end{equation}
Similarly to Remark \ref{Re-SDE-1}, we have the following result.

\begin{remark}\label{Re-BSDE-1}\sl
Let {\bf{(H$_{1}$)}} hold. If the solution $ (y(\cdot),z(\cdot),k(\cdot, \cdot)) $ to MF-BSDE with jumps \eqref{equ-BSDE-1} belongs to $ L_{\mathbb{F}}^{2, K}\left(t, \infty ; \mathbb{R}^{n}\right)   \times L_{\mathbb{F}}^{2, K}\left(t, \infty ; \mathbb{R}^{n d}\right) \times \cK_{\mathbb{F}}^{2, K}\left(t, \infty ;\mathbb{R}^{nl}\right) $, then
$
 \lim\limits_{s \rightarrow \i} \mathbb{E}\big[ |y(s) \mathrm e^{ K s} |^{2}\big]=0.
$
\end{remark}

Before studying the wellposedness of  BSDE \eqref{equ-BSDE-1},  we firstly study the following a priori estimate.

\begin{lemma}\label{Le-BSDE-1}\sl
Assume that {\bf{(H$_{1}$)}}  holds and $f(\cdot) \in L_{\mathbb{F}}^{2,  K_1}(0, \infty ; \mathbb{R}^{n})$ with some $K_1\in\mathbb{R}$.
Let $K$ satisfy $K\leq K_1$,  $K<\kappa $,  and  $ \vartheta(\cdot)\in L_{\mathbb{F}}^{2, K}(t, \infty ; \mathbb{R}^{n}) \times L_{\mathbb{F}}^{2, K}(t, \infty ; \mathbb{R}^{n d}) \times \cK_{\mathbb{F}}^{2, K}(t, \infty ; \mathbb{R}^{n  l})$ be a  solution to MF-BSDE \eqref{equ-BSDE-1}. Then, for any $\varepsilon\in(0, 2\kappa_2 -2K)$, we have
\begin{equation}\label{equ-BSDE-3-1}
\begin{aligned}
&\mathbb{E} |y(t) \mathrm e^{ K t} |^{2}+ (2\kappa -2K-\varepsilon) \mathbb{E} \int_{t}^{\infty}|y(s) \mathrm e^{ K s} |^{2} \mathrm ds+ \mathbb{E}\int_{t}^{\infty}\( |z(s)\mathrm e^{K s}|^{2}+    \|k(s, \cdot)\mathrm e^{K s}\|_\rho^{2}  \)  \mathrm ds\\
&\leq \(L_{\varepsilon, 2} +L_{ \varepsilon,3}+\frac1\varepsilon\) \mathbb{E}\int_{t}^{\infty} |f(s)e^{ K s}|^{2} \mathrm ds,
\end{aligned}
\end{equation}
%
%
where  $L_{\varepsilon ,2}:=\frac{1}{\varepsilon}\( 1+ \frac{\sup\limits_{s\in[0,\infty)}\big(\|\mathbf{C}^{2}(s) \|^{2}+\|  \mathbf{M}^{2}(s, \cdot) \|_\rho^{2}\big) }{\varepsilon}  \) $ and
$L_{\varepsilon,3}:=\frac{ 2\sup\limits_{s\in[0,\infty)}(1 + 2\|\mathbf{C}^{1}(s)\|^{2} + 2\|\mathbf{M}^{1}(s,\cdot)\|_\rho^{2})}{\varepsilon (2\kappa_2-2K-\varepsilon)} + 2 $.

\smallskip
Moreover, let  $\bar{\vartheta}(\cdot) \in L_{\mathbb{F}}^{2, K}(t, \infty ; \mathbb{R}^{n}) \times
L_{\mathbb{F}}^{2, K}(t, \infty ; \mathbb{R}^{n d}) \times \cK_{\mathbb{F}}^{2, K}(t, \infty ;$
$\mathbb{R}^{n l})$ be a solution to the MF-BSDE with jumps \eqref{equ-BSDE-1} with another  nonhomogeneous term    $\bar{f}(\cdot)
\in  L_{\mathbb{F}}^{2,  K_1}(0, \infty ; \mathbb{R}^{n})$.  Then, for any  $\varepsilon\in(0,2\kappa_2 -2K )$, we have
\begin{equation}\label{equ-BSDE-4-1}
\begin{aligned}
&\mathbb{E} |(y(t)-\bar{y}(t)) \mathrm e^{K t} |^{2}+(2\kappa -2K-\varepsilon) \mathbb{E} \int_{t}^{\infty}|(y(s)-\bar{y}(s)) \mathrm e^{ K s} |^{2} \mathrm ds\\
&+  \mathbb{E}\int_{t}^{\infty}\( |(z(s) - \bar{z}(s)) \mathrm e^{K s}|^{2}+ \|(k(s, \cdot)-\bar{k}(s, \cdot)) \mathrm e^{K s}\|_\rho^{2}  \)  \,\mathrm ds \\
&\leq \(L_{\varepsilon, 2} +L_{\varepsilon, 3} +\frac1\varepsilon\) \mathbb{E}\int_{t}^{\infty} |(f(s)-\bar{f}(s))e^{ K s}|^{2} \mathrm ds.
\end{aligned}
\end{equation}
%

\end{lemma}
\begin{proof}
Firstly, we decompose  the solution of MF-BSDE with jumps \eqref{equ-BSDE-1} as follows, $$(y(\cdot),z(\cdot),k(\cdot,\cdot))=(y^{(1)}(\cdot)+y^{(2)}(\cdot),z^{(1)}(\cdot)+z^{(2)}(\cdot),k^{(1)}(\cdot,\cdot)+k^{(2)}(\cdot,\cdot)),$$ where
\begin{equation}\label{equ-BSDE-5}
\left\{\begin{aligned}
&\!\! \mathrm d y^{(1)}(s)\!=\!-\Big\{ A_K(s)^{\top}y^{(1)}(s) + \mathbf{C}^{1}(s)^{\top} z^{(1)}(s)
+ \int_E \mathbf{M}^{1}(s, e)^{\top} \varrho(\mathrm de)k^{(1)}(s, e)\!\\
&  \hskip2cm
+ f(s)-\mathbb{E}{[f(s)]} \Big\}   \mathrm ds
+\sum\limits_{i=1}^dz_i(s) \, \mathrm dW_i(s)
+\sum\limits_{j=1}^l\int_E k_j(s, e)   \widetilde{\mu}_j(\mathrm ds, \mathrm de),\q s\geq t,\!\\
&\!\!\ \mathrm d y^{(2)}(s)\! =\!-\Big\{  (2KI+ A^{2}(s))^{\top} y^{(2)}(s) +\mathbf{C}^{2}(s)^{\top} z^{(2)}(s)+\int_E \!\mathbf{M}^{2}(s, e)^{\top} \varrho(\mathrm de)k^{(2)}(s, e)  +\\
&\hskip2cm
+  \mathbb{E}{[f(s)]} \Big\}   \mathrm ds . \!\!
\end{aligned}\right.
\end{equation}
For any $T >t$ and $K\leq K_1$, applying It\^o's formula   to
$| y^{(1)}(\cdot)  \mathrm e^{K \cdot}|^{2}$ and $|y^{(2)}(\cdot)  \mathrm e^{K \cdot}|^{2}$
on the interval $[t, T]$, respectively, we get
\begin{equation}\label{equ-BSDE-8}
\begin{aligned}
&\mathbb{E}| y^{(1)}(t)  \mathrm e^{K t}|^{2}
=\mathbb{E}| y^{(1)}(T) \mathrm e^{K T}|^{2}
+\mathbb{E}\int_{t}^{T}\Big\{ \langle ( A^{1}(s)+A^{1}(s)^{\top}\! +2KI )y^{(1)}(s), y^{(1)}(s)\rangle \\
&\
+ 2 \langle y^{(1)}(s), \mathbf{C}^{1}(s)^{\top}z^{(1)}(s) +\int_E \mathbf{M}^{1}(s, e)^{\top} \varrho(\mathrm de)k^{(1)}(s, e)+f(s)-\mathbb{E}{[f(s)]} \rangle
 \\
& \  - |z(s)|^2 -\| k(s, \cdot)\|_\rho^2 \Big\} \mathrm e^{2 K s}\,\mathrm ds \\
& =\mathbb{E}| y^{(1)}(T) \mathrm e^{K T}|^{2}+ \mathbb{E}\int_{t}^{T}\Big\{-|\mathbf{C}^{1}(s)  y^{(1)}(s) - z^{(1)}(s)|^{2}
 - \|\mathbf{M}^{1}(s,\cdot)  y^{(1)}(s) - k^{(1)}(s,\cdot)\|_\rho^{2}     \\
& \ +\langle (A^{1}(s)+A^{1}(s)^{\top}+\mathbf{C}^{1}(s)^{\top}\mathbf{C}^{1}(s)  + \int_E \mathbf{M}^{1}(s,e)^\top\varrho(\mathrm de)\mathbf{M}^{1}(s,e) +2KI) y^{(1)}(s), y^{(1)}(s)\rangle\!\!\!\!\\
&\ + 2\langle y^{(1)}(s), f(s) \rangle  - | z^{(2)}(s)|^2   - \|k^{(2)}(s, \cdot)\|_\rho^2 \Big\}\mathrm e^{2 K s}\mathrm ds\\
& \leq \mathbb{E}| y^{(1)}(T) \mathrm e^{K T}|^{2}+ \mathbb{E}\int_{t}^{T}\Big\{ \langle (-2\kappa_2+2K ) y^{(1)}(s), y^{(1)}(s)\rangle + 2\langle y^{(1)}(s), f(s) \rangle   \Big\}\mathrm e^{2 K s}\mathrm ds\\
 &\ -  \mathbb{E}\int_{t}^{T}\Big\{| z^{(2)}(s)\mathrm e^{K s}|^2 + \|k^{(2)}(s, \cdot)\mathrm e^{K s}\|_\rho^2 \Big\}\mathrm ds,
%
%
%
\end{aligned}
\end{equation}
and
\begin{equation}\label{equ-BSDE-9}
\begin{aligned}
&| y^{(2)}(t) \mathrm e^{K t}|^{2}
 =|y^{(2)}(T) \mathrm e^{K T}|^{2}
+\int_{t}^{T}\Big\{ \langle (A^{2}(s)+ A^{2}(s)^{\top}\!+2KI )y^{(2)}(s),  y^{(2)}(s)\rangle
 \\
&\quad
+2 \langle y^{(2)}(s), \mathbf{C}^{2}(s)^{\top} z^{(2)}(s)+\int_E \mathbf{M}^{2}(s, e)^{\top} \varrho(\mathrm de)k^{(2)}(s, e)+ \mathbb{E}{[f(s)]}\rangle \Big\} \mathrm e^{2 K s}\mathrm ds\\
 &\leq |y^{(2)}(T) \mathrm e^{K T}|^{2}
+\int_{t}^{T}\Big\{ \langle (-2\kappa_1+2K )y^{(2)}(s),  y^{(2)}(s)\rangle
  \\
&\quad
+2 \langle y^{(2)}(s),\mathbf{C}^{2}(s)^{\top} z^{(2)}(s)+ \int_E \mathbf{M}^{2}(s, e)^{\top} \varrho(\mathrm de)k^{(2)}(s, e)+ \mathbb{E}{[f(s)]}\rangle \Big\} \mathrm e^{2 K s}\mathrm ds .
\end{aligned}
\end{equation}
Then, from  \eqref{equ-BSDE-8}, for any $\varepsilon>0$, we get
\begin{equation}\label{equ-BSDE-12}
\begin{aligned}
&\mathbb{E}|y^{(1)}(t) \mathrm e^{K t}|^{2}+ (2\kappa_2-2K-\varepsilon) \mathbb{E}\int_{t}^{T}|y^{(1)}(s) \mathrm e^{ K s}|^2\mathrm ds\\
&+ \int_{t}^{T}\Big\{| z^{(2)}(s)\mathrm e^{K s}|^2 +\|k^{(2)}(s,\cdot)\mathrm e^{K s}\|_\rho^2 \Big\}\mathrm ds \leq  \mathbb{E}|y^{(1)}(T) \mathrm e^{K T}|^{2}   +   \frac1\varepsilon \mathbb{E}\int_{t}^{T} | f(s)     \mathrm e^{K s}|^2 \mathrm ds.
\end{aligned}
\end{equation}

Note that, for any $s\geq t$, using the H\"{o}lder  inequality and Cauchy inequality, we have
$$
\begin{aligned}
&
 2 \langle y^{(2)}(s) , \int_E \mathbf{M}^{2}(s, e)^{\top} \varrho(\mathrm de)k^{(2)}(s, e)\rangle
 %
 \leq \varepsilon|y^{(2)}(s)| ^2+\frac 1\varepsilon\Big| \int_E \mathbf{M}^{2}(s, e)^{\top} \varrho(\mathrm de)k^{(2)}(s, e)\Big|^2\\
 &\leq \varepsilon|y^{(2)}(s)|^2+\frac 1\varepsilon\Big| \sum\limits_{j=1}^l\(\int_E  \|\mathbf{M}^{2}_j(s,e)  \|^2\rho_j(\mathrm de)  \)^\frac12 \cdot\( \int_E  | k^{(2)}_j(s, e) |^2\rho_j(\mathrm de)\)^\frac12\Big|^2\\
%
 &\leq \varepsilon|y^{(2)}(s)| ^2 +\frac 1\varepsilon  \|  \mathbf{M}^{2}(s, \cdot) \|_\rho^{2}\cdot\|k^{(2)}(s,\cdot)\|_\rho^{2}.
\end{aligned}
$$
%
Then, from  \eqref{equ-BSDE-9}, when $2\kappa_2-2K- \varepsilon>0$, we  get
\begin{equation}\label{equ-BSDE-13}
\begin{aligned}
&|y^{(2)}(t) \mathrm e^{K t}|^{2}
+(2\kappa_1-2K-3 \varepsilon)\int_{t}^{T} |y^{(2)}(s) \mathrm e^{ K s}|^2\mathrm ds \\
&\!\leq \!|y^{(2)}(T) \mathrm e^{K T}|^{2}
+\frac1\varepsilon\!\int_{t}^{T}\!\Big\{ \|\mathbf{C}^{2}(s)\|^{2} |z^{(2)}(s) |^{2}
+ \|  \mathbf{M}^{2}(s, \cdot) \|_\rho^{2}\!\cdot\!\|k^{(2)}(s,\cdot)\|_\rho^{2}
+|\mathbb{E}{[f(s)]} |^2 \Big \}  \mathrm e^{2 K s}\mathrm ds \!\!\\
&\!\leq\! \varepsilon L_{\varepsilon,2} |y(T) \mathrm e^{K T}|^{2} +L_{\varepsilon,2} \int_{t}^{T} \mathbb{E}  |f(s)  \mathrm e^{ K s}|^{2} \mathrm ds,
\end{aligned}
\end{equation}
where $L_{\varepsilon ,2}:=\frac{1}{\varepsilon}\( 1+ \frac{\sup\limits_{s\in[0,\infty)}\{\|\mathbf{C}^{2}(s) \|^{2}+\|  \mathbf{M}^{2}(s, \cdot) \|_\rho^{2}\} }{\varepsilon}  \) $.

Combining inequalities \eqref{equ-BSDE-12}, \eqref{equ-BSDE-13} and letting $ T \to \infty $,  we   obtain
\begin{equation}\label{y1+y2}
\begin{aligned}
&\mathbb{E} |y(t) \mathrm e^{ K t} |^{2}+(2\kappa -2K-3\varepsilon) \mathbb{E} \int_{t}^{\infty}|y(s) \mathrm e^{ K s} |^{2} \mathrm ds\\
&\leq \mathbb{E} |y^{(1)}(t)e^{K t}|^{2}+(2\kappa_2-2K- \varepsilon) \mathbb{E} \int_{t}^{\infty} |y^{(1)}(s) \mathrm e^{ K s} |^{2}  \mathrm ds\\
&\q+ |y^{(2)}(t)e^{K t}|^{2} +(2\kappa_1-2K-3\varepsilon) \int_{t}^{\infty} |y^{(2)}(s) \mathrm e^{ K s} |^{2}\mathrm ds\\
&\leq \(L_{\varepsilon, 2}  +\frac1\varepsilon\) \mathbb{E}\int_{t}^{\infty} |f(s)e^{ K s}|^{2}  \mathrm ds  .
\end{aligned}
\end{equation}
 On the other hand,  from the first equality in  \eqref{equ-BSDE-8}, we have
$$
\begin{aligned}
&\mathbb{E}| y^{(1)}(t) \mathrm e^{K t}|^{2}
+ \mathbb{E}\int_{t}^{T}\( |z(s) \mathrm e^{ K s}|^{2}+ \|k(s, \cdot) \mathrm e^{ K s}\|_\rho^{2} \)\mathrm ds \\
&=\mathbb{E}|y^{(1)}(T)  \mathrm e^{K T}|^{2}
+\mathbb{E}\int_{t}^{T}\Big\{ \langle (A^{1}(s)+A^{1}(s)^{\top}\!+\!2KI )y^{(1)}(s), y^{(1)}(s)\rangle
 \\
&\quad
+ 2\langle y^{(1)}(s),\mathbf{C}^{1}(s)^{\top}z^{(1)}(s) +\int_E \mathbf{M}^{1}(s, e)^{\top} \varrho(\mathrm de)k^{(1)}(s, e) + f(s) \rangle \Big\} \mathrm e^{2 K s}\mathrm ds\\
&\leq \mathbb{E}|y^{(1)}(T) \mathrm e^{K T}|^{2}
+ \mathbb{E}\int_{t}^{T}\Big\{( 1+  2\|\mathbf{C}^{1}(s)\|^{2} + 2 \|\mathbf{M}^{1}(s, \cdot)\|_\rho^{2}) |y^{(1)}(s)|^{2}+
 |f(s)|^{2}
 \\
&\quad +\frac{1}{2} |z^{(1)}(s)|^{2} +\frac{1}{2}  \|k^{(1)}(s,\cdot)\|_\rho^{2}
 \Big\} \mathrm e^{2 K s}\mathrm ds \\
&\leq \mathbb{E}|y(T) \mathrm e^{K T}|^{2}
+ \mathbb{E}\int_{t}^{T}\Big\{( 1+ 2 \|\mathbf{C}^{1}(s)\|^{2} + 2 \|\mathbf{M}^{1}(s, \cdot)\|_\rho^{2}) |y^{(1)}(s)|^{2}+
 |f(s)|^{2}
 \\
&\quad+\frac{1}{2} |z(s)|^{2}  +\frac{1}{2} \|k(s,\cdot)\|_\rho^{2}
 \Big\} \mathrm e^{2 K s}\mathrm ds,
\end{aligned}
$$
where we have used  $ K <  \kappa_2$ and $\mathbb{E} |z^{(1)}(\cdot)|^{2} \leq \mathbb{E} |z(\cdot)|^{2} $,
$\displaystyle \mathbb{E}\|k^{(1)}(\cdot,\cdot)\|_\rho^{2}\leq \mathbb{E}\|k(\cdot,\cdot)\|_\rho^{2}$.

Letting $ T \to \infty $ and using \eqref{equ-BSDE-12}, when $\varepsilon\in(0,2\kappa_2-2K)$, we get
\begin{equation}\label{equ-BSDE-15}
\begin{aligned}
& \mathbb{E}\int_{t}^{\infty}\( |z(s)\mathrm e^{K s}|^{2}+ \|k(s, \cdot)\mathrm e^{K s}\|_\rho^{2}\) \mathrm ds  \\
&\leq  2\mathbb{E}\int_{t}^{\infty}    \(( 1+2\|\mathbf{C}^{1}(s)\|^{2} +2 \|\mathbf{M}^{1}(s,\cdot)\|_\rho^{2})|y^{(1)}(s)|^{2} +  |f(s)|^{2} \)\mathrm e^{2 K s}\mathrm ds \\
&\leq L_{\varepsilon,3} \mathbb{E}\int_{t}^{\infty} |f(s) \mathrm e^{  K  s}|^{2}  \mathrm ds,
\end{aligned}
\end{equation}
with $L_{\varepsilon, 3}:= \frac{ 2\sup\limits_{s\in[0,\infty)}(1 + 2\|\mathbf{C}^{1}(s)\|^{2} + 2\|\mathbf{M}^{1}(s,\cdot)\|_\rho^{2})}{\varepsilon (2\kappa_2-2K-\varepsilon)} + 2 $.\par
Then,  \eqref{equ-BSDE-3-1} is derived   by combing  \eqref{y1+y2} with \eqref{equ-BSDE-15}.
Further, \eqref{equ-BSDE-4-1} is  obvious by using   the equation  of $ (y(\cdot)-\bar{y}(\cdot) ,z(\cdot)-\bar{z}(\cdot) ,k(\cdot,\cdot)-\bar{k}(\cdot,\cdot) )$ and its corresponding  estimate like \eqref{equ-BSDE-3-1}.

\end{proof}
\begin{lemma}\label{Le-BSDE-2}\sl
Assume {\bf{(H$_{1}$)}}  holds and $f(\cdot) \in L_{\mathbb{F}}^{2,  K_1}(0, \infty ; \mathbb{R}^{n})$ with $K_1\in\mathbb{R}$.
Let   $K\leq K_1$ and  $K<\kappa $,
Then, MF-BSDE with jumps  \eqref{equ-BSDE-1} admits a unique solution $(y(\cdot), z(\cdot), k(\cdot, \cd))\in L_{\mathbb{F}}^{2, K}\left(t, \infty ; \mathbb{R}^{n}\right) \times L_{\mathbb{F}}^{2, K}\left(t, \infty ; \mathbb{R}^{n d}\right) \times \cK_{\mathbb{F}}^{2, K}(t, \infty ;$ $\left.\mathbb{R}^{n l}\right)$.
\end{lemma}
The a priori estimates \eqref{equ-BSDE-3-1}  can bring us the uniqueness of the solution to  MF-BSDE with jumps \eqref{equ-BSDE-1}. The proof of the existence is similar to the method in \cite{Peng-Shi} so that we omit it here, also referring to \cite{Shi-Zhao-2020, Wei-Yu-2021, Yu-2017}.

\section{Infinite Horizon Mean-Field  Linear Quadratic  Optimal Control Problem with Time-varying Coefficients} 
\label{sec:LQ}
 In this section, we study the infinite horizon mean-field linear quadratic  optimal control problem with time-varying coefficients. For any initial pair $(t,x_t)\in[0,\infty)\times L_{\mathcal{F}_t}^2(\Omega;\mathbb{R}^n)$,  the state equation \eqref{state}
and the   cost functional \eqref{cost} with  $K \in \mathbb{R}$, $u(\cd)\in L_\mathbb{F}^{2,K}(0,\infty;\mathbb{R}^m)$ will be considered.
For the coefficients $A$, $\bar A,$    $C_i$, $ \bar{C}_i,$  $M_j$, $ \bar{M}_j$ $(i=1,\cdots, d,$ $j=1,\cdots,l)$ of \eqref{state},   we still assume {\bf (H$_{1}$)}. In addition,  we also need the following condition  on the other  coefficients.

\ms
\noindent \textbf{Assumption (H$_{2}$)}
\!\! \! $B(\cdot), \bar{B}(\cdot)\in L^{\infty}\left(0, \infty ; \mathbb{R}^{n \times m}\right)$,
$\mathbf{D}(\cdot),  \bar{\mathbf{D}}(\cdot) \in L^{\infty}\left(0, \infty ; \mathbb{R}^{(n d) \times m}\right)$,
$\mathbf{N}(\cdot,\cdot), \bar {\mathbf{N}}(\cdot,\cdot)  \in L^{\infty}\left(0, \infty ; L_\rho^2\left(E;  \mathbb{R}^{(nl) \times m}\right)\right)$,
$Q(\cdot), \bar{Q}(\cdot)\in L^{\infty}\left(0, \infty ; \dbS^{n}\right)$,
$R(\cdot),  \bar{R}(\cdot)\in L^{\infty}\left(0, \infty ; \dbS^{m}\right)$  and
$S(\cdot),  \bar{S}(\cdot)\in L^{\infty}\left(0, \infty ; \mathbb{R}^{m \times n}\right)$,
where
$\mathbf{D}(\cdot)=(D_1(\cdot)^\top,\dots,D_d(\cdot)^\top)^\top, $ $\bar{\mathbf{D}}(\cdot)=(\bar D_1(\cdot)^\top,\dots,\bar D_d(\cdot)^\top)^\top $,
$\mathbf{N}(\cdot,\cdot)=(N_1(\cdot,\cdot)^\top,\dots,N_l(\cdot,\cdot)^\top)^\top$, $\bar {\mathbf{N}}(\cdot,\cdot)=(\bar N_1(\cdot,\cdot)^\top,\dots,\bar N_l(\cdot,\cdot)^\top)^\top  $.

\smallskip
%
 Let us  assume {\bf(H$_{1}$)}, {\bf(H$_{2}$)} and  $K<\kappa $ with $\kappa $  here being the one in \eqref{kappa}.  By Lemma \ref{Le-SDE-1}, for any initial pair $(t,x_t)\in[0,\infty)\times L_{\mathcal{F}_t}^2(\Omega;\mathbb{R}^n)$ and $u(\cdot)\in L_{\mathbb{F}}^{2, K}(t, \infty ; \mathbb{R}^{m}) $, the state equation  \eqref{state} admits a unique solution
$x(\cdot) \equiv x(\cdot; t, x_{t}, u(\cdot)) \in L_{\mathbb{F}}^{2, K}(t, \infty ; \mathbb{R}^{n})$. This also ensures
the cost functional $J^{K}\left(t, x_{t} ; u(\cdot)\right) $ in \eqref{cost} to make sense.

Denote  $\mathcal{U}^{K}[t, \infty):=L_{\mathbb{F}}^{2, K}\left(t, \infty ; \mathbb{R}^{m}\right)$, then we call  $u(\cdot)\in \mathcal {U}^K[t,\infty)$ as the admissible control. Sequentially,
  $x(\cdot ; t, x_{t}, u(\cdot))$, $(u(\cdot), x(\cdot ; t, x_{t}, u(\cdot)))$ are said to be     the admissible   state and the admissible pair, respectively.
Based on these preparations,    Problem (MF-LQ) formulated in Introduction   makes sense.
%

%
%
 \begin{definition}\sl
If there exists a (unique) admissible control $u^{*}(\cdot)$ satisfying \eqref{equ-MFLQ-5},  Problem (MF-LQ) is said to be  (uniquely) open-loop solvability at $\left(t, x_{t}\right)$. Such $u^{*}(\cdot)$   is called an open-loop optimal control of Problem (MF-LQ) at $\left(t, x_{t}\right)$,
the state $x^{*}(\cdot) \equiv x\left(\cd ; t, x_{t}, u^{*}(\cdot)\right)$ is called the corresponding optimal state process,
and $\left(u^{*}(\cdot), x^{*}(\cdot)\right)$ is called an optimal pair of Problem (MF-LQ) at $\left(t, x_{t}\right)$.\par
\end{definition}

To characterize the open-loop optimal control  of  Problem (MF-LQ), we  assume the following positive definiteness condition.

\ss
\noindent$\textbf{Condition (PD).}$
$
R^1(\cdot) \gg 0, \q  R^2(\cdot) \gg 0,\q
\left(\begin{array}{cc}
Q^1(\cdot) & S^1(\cdot)^{\top} \\
S^1(\cdot) & R^1(\cdot)
\end{array}\right) \geq 0, \q
 \left(\begin{array}{cc}
 Q^2(\cdot) &  S^2(\cdot)^{\top} \\
 S^2(\cdot) &  R^2(\cdot)
\end{array}\right) \geq 0.
$

\ms

 Note that, the Schur's lemma implies  Condition {\bf (PD)} to be  equivalent to
\begin{equation}\label{PD-1}
R^1(\cdot) \gg 0, \q  R^2(\cdot) \gg 0,\q Q^1(\cdot)-S^1(\cdot)^{\top} R^1(\cdot)^{-1} S^1(\cdot) \geq 0,\q  Q^2(\cdot)-S^2(\cdot)^{\top}  R^2(\cdot)^{-1}  S^2(\cdot) \geq 0.
\end{equation}
%

The following is one of the main results, which works as the usual Pontryagin maximum principle.

\begin{lemma}\label{Le-MFLQ-1}\sl
Assume  {\bf{(H$_{1}$)}}, {\bf{(H$_{2}$)}} and Condition {\bf (PD)} hold. Given $K<\kappa $,
let $\left(t, x_{t}\right) \in[0, \infty) \times$ $L_{\mathcal{F}_{t}}^{2}\left(\Omega ; \mathbb{R}^{n}\right)$
be a pair of  initial data and
$\left(u^{*}(\cdot), x^{*}(\cdot)\right) \in \mathcal{U} ^{K}[t, \infty) \times L_{\mathbb{F}}^{2, K}\left(t, \infty ; \mathbb{R}^{n}\right)$ be an admissible pair.
Then,
  $u^{*}(\cdot)$ is an optimal control of Problem (MF-LQ) at $\left(t, x_{t}\right)$  if and only if
 $u^{*}(\cdot)$ satisfies
\begin{equation}\label{optimal}
\begin{aligned}
&B(\cdot)^{\top} y^{*}(\cdot) + \bar{B}(\cdot)^{\top} \mathbb{E}{[y^{*}(\cdot)]} + \mathbf{D}(\cdot)^{\top} z^{*}(\cdot)
+\bar{\mathbf{D}} (\cdot)^{\top} \mathbb{E}{[z^{*} (\cdot)]} + S(\cdot) x^{*}(\cdot) + \bar{S}(\cdot)\mathbb{E}{[x^{*}(\cdot)]}\\
&+\ds  \int_E\! \big(\mathbf{N} (\cdot, e)^{\top} \varrho(\mathrm de) k^{*}(\cdot,e)
+\bar{\mathbf{N}}(\cdot,e)^{\top} \varrho(\mathrm de)\mathbb{E}{[k^{*}(\cdot, e)]}\big)
+R(\cdot) u^{*}(\cdot) + \bar{R}(\cdot)\mathbb{E}{[u^{*}(\cdot)]}=0, \ \mathbb{P}\mbox{-a.s.}\!\!\!\!\!
\end{aligned}
\end{equation}
where $\left(y^{*}(\cdot), z^{*}(\cdot), k^{*}(\cdot, \cd)\right) \in L_{\mathbb{F}}^{2, K}\left(t, \infty ; \mathbb{R}^{n}\right) \times L_{\mathbb{F}}^{2, K}\left(t, \infty ; \mathbb{R}^{n d}\right)\times \cK_{\mathbb{F}}^{2, K}\left(t, \infty ; \mathbb{R}^{n l}\right)$ is the solution to  the following infinite horizon MF-BSDE with jumps
\begin{equation}\label{adjoint}
\begin{aligned}
&\mathrm dy^{*}(s)= -\Big\{ A_K(s)^{\top}  y^{*}(s)
                    +\bar{A}(s)^{\top}\mathbb{E}{[y^{*}(s)]} +   \mathbf{C}(s) ^{\top} z^{*}(s)+\bar {\mathbf{C}}(s)^{\top}\mathbb{E}{[z^{*}(s)]}   \\
&     + \int_E \big(\mathbf{M} (s,e) ^ {\top}\varrho (\mathrm de) k^{*}(s, e)
                    +\bar{\mathbf{M}} (s, e)^{\top} \varrho (\mathrm de) \mathbb{E}{[k^{*}(s, e)]}\big) +Q(s) x^{*}(s) + \bar{Q}(s)\mathbb{E}{[x^{*}(s)]}  \\
&    +S(s)^{\top} u^{*}(s)+\bar{S}(s)^{\top}\mathbb{E}{[u^{*}(s)]}\Big\} \, \mathrm ds
                    + \sum\limits_{i=1}^d  z_i(s) \, \mathrm dW_i(s)\!+\! \sum\limits_{j=1}^l \!\int_E k_j(s, e)   \widetilde{\mu}_j(\mathrm ds, \mathrm de), \quad   s \geq t.
\end{aligned}
\end{equation}
\end{lemma}

\begin{proof}
 For any $u(\cdot) \in \mathcal{U} ^{K}[t, \infty)$ and any $\epsilon \in \mathbb{R}$, denote $x^{\epsilon}(\cdot) \in L_{\mathbb{F}}^{2, K}\left(t, \infty ; \mathbb{R}^{n}\right)$  by  the  solution  to \eqref{state} under $u^{*}(\cdot)+\epsilon u(\cdot) \in$ $\mathcal{U} ^{K}[t, \infty)$. Then $x^0(s):=\frac{x^{\epsilon}(s)-x^{*}(s)}{\epsilon}$, $s\geq t$ satisfies
\begin{equation}\label{equ-MFLQ-10}
\left\{
\begin{aligned}
& \!\! \mathrm dx^0(s)=\big(A(s) x^0(s)+\bar{A}(s)\mathbb{E}{[x^0(s)]}+B(s) u(s)+\bar{B}(s)\mathbb{E}{[u(s)]}\big)\, \mathrm ds \\
&  \!\!   +\sum\limits_{i=1}^d\big(C_i(s) x^0(s)\!+\!\bar{C}_i(s)\mathbb{E}{[x^0(s)]}\!+\!D_i(s) u(s)\!+\!\bar{D}_i(s)\mathbb{E}{[u(s)]}\big) \, \mathrm dW_i(s) \\
&  \!\!  +\! \ds\sum\limits_{j=1}^l\! \int_E\!\! \big(M_j(s,e) x^0(s-)\!+\!\bar{M}_j(s, e)\mathbb{E}{[x^0(s-)]}\!+\!N_j(s, e) u(s)
\!+\!\bar{N}_j(s, e)\mathbb{E}{[u(s)]}\big)  \widetilde{\mu}_j\left(\mathrm ds, \mathrm de\right),\!\! \!\! \!\! \!\!   \\
&\!\! x^0(t)=0 .
\end{aligned}
\right.
\end{equation}
By  applying It\^o's formula to $\left\langle e^{K \cdot} x^0(\cdot), e^{K \cdot}
y^{*}(\cdot)\right\rangle$   and using $\lim\limits _{T \rightarrow \i} \mathbb{E}\left\langle \mathrm e^{K T} x^0(T), \mathrm e^{K T} y^{*}(T)\right\rangle=0$, we get (omitting the  variables $s$  and $e$ when without causing confusions)
%
%
$$
\begin{aligned}
&\mathbb{E}\!\int_{t}^{\infty}\! e^{2 K s}\big\langle x^0, Q x^{*}+\bar{Q}\mathbb{E}{[x^{*}]}
+S^{\top} u^{*} + \bar{S}^{\top}\mathbb{E}{[u^{*}]}\big\rangle\mathrm ds\\
&\!\!=\mathbb{E}\! \int_{t}^{T} \!\!\!e^{2 K s}\big\langle B^{\top} y^{*}\!+\!\bar{B}^{\top}\mathbb{E}{[y^{*}]}
\!+\! \mathbf{D}^{\top} z^{*}\!+\!\bar{\mathbf{D}}^{\top}\mathbb{E}{[z^{*}]} \! +\! \int_E \!\! \big(\mathbf{N}(e)^{\top} \varrho(\mathrm de) k^{*}(e)
\!+\!\bar{\mathbf{N}}(e)^{\top} \varrho(\mathrm de)\mathbb{E}{[k^{*}(e)]}\big) ,  u^{*}\big\rangle \,\mathrm ds .
\end{aligned}
$$

On the other hand,
$$
\begin{array}{lll}
&J^{K}\left(t, x_{t} ; u^{*}(\cdot)+\epsilon  u(\cdot)\right)-J^{K}\left(t, x_{t} ; u^{*}(\cdot)\right) \\
&=\epsilon  \mathbb{E}\ds\int_{t}^{\infty} \!e^{2 K s}\Big(\big\langle Q x^{*}+\bar{Q}\mathbb{E}{[x^{*}]}
+S^{\top} u^{*}+\bar{S}^{\top}\mathbb{E}{[u^{*}]}, x^0\big\rangle +\big\langle S x^{*}
+\bar{S}\mathbb{E}{[x^{*}]}+R u^{*}+\bar{R}\mathbb{E}{[u^{*}]}, u\big\rangle\Big) \,\mathrm ds \\
&\  +\dfrac{\epsilon ^{2}}{2} \mathbb{E}\ds\!\int_{t}^{\infty} \!e^{2 K s}g\big(s,x^0(s),\mathbb{E}[x^0(s)],u(s),\mathbb{E}[u(s)]\big)  \, \mathrm ds.
\end{array}
$$
Combining  the above two equalities,
$$
\begin{array}{lll}
&J^{K}\left(t, x_{t} ; u^{*}(\cdot)+\epsilon u(\cdot)\right)-J^{K}\left(t, x_{t} ; u^{*}(\cdot)\right) \\
&\!=\epsilon \mathbb{E}\!\!\!\ds\int_{t}^{\infty} e^{2 K s}\Big(\big\langle B^{\top} y^{*}
\!+\!\bar{B}^{\top}\mathbb{E}{[y^{*}]}\!+\! \mathbf{D}^{\top} z^{*}\!+\!\bar{\mathbf{D}}^{\top}\mathbb{E}{[z^{*}]}+\int_E  \big(\mathbf{N}(e)^{\top} \varrho(\mathrm de) k^{*}(e)
\!+\!\bar{\mathbf{N}}(e)^{\top} \varrho(\mathrm de)\mathbb{E}{[k^{*}(e)]}\big) \\
& \displaystyle\hskip2.7cm +S x^{*}+\bar{S}\mathbb{E}{[x^{*}]}+R u^{*}+\bar{R}\mathbb{E}{[u^{*}]} , u\big\rangle\Big)\, \mathrm ds \\
&\ +\dfrac{\epsilon ^{2}}{2} \mathbb{E}\ds\int_{t}^{\infty} e^{2 K s}g\big(s,x^0(s),\mathbb{E}[x^0(s)],u(s),\mathbb{E}[u(s)]\big) \,\mathrm ds.
\end{array}
$$
It is natural that, under Condition {\bf (PD)},  $u^*(\cdot)$ is an optimal control of Problem (MF-LQ) at $(t, x_t)$ if and
only if \eqref{optimal} holds.

\end{proof}
\begin{remark}\label{Re-MFLQ-1}\sl  The infinite horizon  MF-BSDE with jumps \eqref{adjoint} is called the adjoint equation of Problem (MF-LQ), whose solution processes $\left(y^{*}(\cdot), z^{*}(\cdot), k^{*}(\cdot, \cd)\right)$ are called the adjoint processes.
  In fact, when  $\left(u^{*}(\cdot), x^{*}(\cdot)\right) \in \mathcal{U} ^{K}[t, \infty) \times L_{\mathbb{F}}^{2, K}\left(t, \infty ; \mathbb{R}^{n}\right)$ and $K<\kappa,$  Lemma \ref{Le-BSDE-2} implies  \eqref{adjoint} to admit a unique solution
$ (y^{*}(\cdot), z^{*}(\cdot),$ $ k^{*}(\cdot, \cdot) )\in  L_{\mathbb{F}}^{2, K} (t, \infty ; \mathbb{R}^{n}) \times L_{\mathbb{F}}^{2, K}(t, \infty ; \mathbb{R}^{n d} ) \times \cK_{\mathbb{F}}^{2, K}(t, \infty ; \mathbb{R}^{n l} )$.
\end{remark}

Let's go back to the optimality condition \eqref{optimal}. Note that, Condition {\bf (PD)} implies that $R^1(\cdot), {R}^2(\cdot)$ are invertible for almost all $s \in[0, \infty)$.
Therefore, we can express the optimal control  $u^{*}(\cdot)$  by virtue of the optimal state $x^*(\cdot)$ and the adjoint processes $\left(y^{*}(\cdot), z^{*}(\cdot), k^{*}(\cdot, \cd)\right)$ as follows, (ignoring the superscript $^*$ from here to save the spaces)
\begin{equation}\label{optimal-u}
\begin{aligned}
&u^{*}(\cdot) =-R^1(\cdot)^{-1}\big( \mathbf{\Lambda}^1[y^{(1)},z^{(1)},k^{(1)}](\cdot)   + S^1(\cdot)  x^{(1)}(\cdot)\big)\\
&\qquad\quad -R^2(\cdot)^{-1}(\mathbf{\Lambda}^2[y^{(2)},z^{(2)},k^{(2)}](\cdot) + S^2(\cdot)  x^{(2)}(\cdot)\big),
\end{aligned}
\end{equation}
where $$\displaystyle   \mathbf{\Lambda}^ \iota  [y^{(\iota)} ,z^{(\iota)} ,k^{(\iota)} ] (\cdot):= B^\iota(\cdot)^{\top}   y^{(\iota)}(\cdot) +   \mathbf{D}^\iota(\cdot)^{\top}   z^{(\iota)}(\cdot)
 +\int_E  \mathbf{N}^\iota(\cdot,e) ^{\top}  \varrho (\mathrm de)  k^{(\iota)}(\cdot,e) ,\quad \iota=1,2.$$
Substituting the above $u^{*}(\cdot)$ into the state equation \eqref{state} and the adjoint equation \eqref{adjoint}, we get 
 %
\begin{equation}\label{Hamil}
\left\{
\begin{aligned}
&\!\! \mathrm dx (s)\!=\! \sum_{ \iota=1}^2\Big((A^\iota  -B^\iota   (R^\iota )^{-1} S^\iota) x^{(\iota)}    -B^\iota   (R^\iota )^{-1}\mathbf{\Lambda}^\iota [y^{(\iota)},z^{(\iota)},k^{(\iota)}] \Big) \, \mathrm ds \\
&    +\sum\limits_{i=1}^d\sum_{ \iota=1}^2\Big((C_i^\iota-D_i^\iota(R^\iota )^{-1} S^\iota) x^{(\iota)} \!-\!D_i^\iota(R^\iota )^{-1}\mathbf{\Lambda}^\iota[y^{(\iota)},z^{(\iota)},k^{(\iota)}]\Big)   \, \mathrm dW_i(s) \\
&  +\sum\limits_{j=1}^l\ds\int_E\sum_{ \iota=1}^2 \Big((M_j^\iota(e) -N_j^\iota(e)  (R^\iota )^{-1} S^\iota) x^{(\iota)}  -N_j^\iota(e) (R^\iota )^{-1}\mathbf{\Lambda}^\iota[y^{(\iota)},z^{(\iota)},k^{(\iota)}]\Big)  \,\widetilde{\mu}_j(\mathrm ds, \mathrm de),  \\
&\!\!  \mathrm dy (s)\! =\!-\sum_{\iota=1}^2\Big( (Q^\iota -(S^\iota)^{\top} (R^\iota )^{-1} S^\iota)x^{(\iota)} +(2 K I+A^\iota -B^\iota  (R^\iota )^{-1} S^\iota )^{\top}y^{(\iota)}
 \\
&   +(\mathbf{C}^\iota -\mathbf{D}^\iota  (R^\iota )^{-1} S^\iota )^{\top} z^{(\iota)}+ \int_E  (\mathbf{M}^{\iota}(e) -\mathbf{N}^\iota(e)  (R^\iota)^{-1} S^\iota )^{\top} \varrho(\mathrm de)  k^{(\iota)} (e) \Big) \, \mathrm d s \\
& + \sum\limits_{i=1}^d  z_i(s) \, \mathrm dW_i(s)\!+\! \sum\limits_{j=1}^l \!\int_E k_j(s, e)   \,\widetilde{\mu}_j(\mathrm ds, \mathrm de),    \quad s \geq t,    \\
&\!\!x (t)\!=\!x_{t},
\end{aligned}\right.
\end{equation}
which is  called as the Hamiltonian system of Problem (MF-LQ) at $(t, x_{t})$.
It is easy to check that \eqref{Hamil} can be rewritten into the following two FBSDEs, which will have a more intuitive
  looking,
\begin{equation}\label{Hamil-x-1}
\left\{
\begin{aligned}
&\!\! \mathrm dx ^{(1)}(s)\!=\!  \Big((A^{1}  -B^1   (R^1 )^{-1} S^1) x^{(1)}   -B^1   (R^1 )^{-1}\mathbf{\Lambda}^1[y^{(1)},z^{(1)},k^{(1)}]\Big) \, \mathrm ds \\
&    +\sum\limits_{i=1}^d\sum_{ \iota=1}^2\Big((C_i^\iota-D_i^\iota(R^\iota )^{-1} S^\iota) x^{(\iota)} \!-\!D_i^\iota(R^\iota )^{-1}\mathbf{\Lambda}^\iota[y^{(\iota)},z^{(\iota)},k^{(\iota)}]\Big)   \, \mathrm dW_i(s)\\
&   +\sum\limits_{j=1}^l\ds\int_E\sum_{\iota=1}^2 \Big((M_j^\iota(e) -N_j^\iota(e)  (R^\iota )^{-1} S^\iota) x^{(\iota)}  -N_j^\iota(e) (R^\iota )^{-1}\mathbf{\Lambda}^\iota[y^{(\iota)},z^{(\iota)},k^{(\iota)}]\Big)\,  \widetilde{\mu}_j(\mathrm ds, \mathrm de), \\
&\!\! \mathrm dy^{(1)} (s)\! =\!- \Big( (Q^1 -(S^1)^{\top} (R^1 )^{-1} S^1)x^{(1)} +(2 K I+A^{1}  -B^1 (R^1 )^{-1} S^1)^{\top}y^{(1)}
  \\
&
+(\mathbf{C}^{1} -\mathbf{D}^1  (R^1 )^{-1} S^1 )^{\top} z^{(1)} + \int_E  (\mathbf{M}^{1}(e) -\mathbf{N}^1 (e) (R^1 )^{-1} S^1 )^{\top} \varrho(\mathrm de)  k^{(1)} (e) \Big) \, \mathrm d s\\
& + \sum\limits_{i=1}^d  z_i(s) \, \mathrm dW_i(s)\!+\! \sum\limits_{j=1}^l \!\int_E k_j(s, e)   \,\widetilde{\mu}_j(\mathrm ds, \mathrm de), \quad s\geq t,     \\
&\!\!x^{(1)} (t)\!=\!x_{t}-\mathbb{E}[x_t],
\end{aligned}\right.
\end{equation}
and
\begin{equation}\label{Hamil-x-2}
\left\{
\begin{aligned}
&\!\! \mathrm dx ^{(2)}(s)\!=\!  \Big((A^{2}  -B^2   (R^2 )^{-1} S^2) x^{(2)}    -B^2   (R^2 )^{-1}\mathbf{\Lambda}^2[y^{(2)},z^{(2)},k^{(2)}]\Big)   \mathrm ds, \quad s\geq t,\\
&\!\!\mathrm dy^{(2)} (s)\! =\!- \Big( (Q^2 -(S^2)^{\top} (R^2 )^{-1} S^2)x^{(2)} +(2 K I+A^{2}  -B^2 (R^2)^{-1} S^2)^{\top}y^{(2)} \\
&
+ (\mathbf{C}^{2} -\mathbf{D}^2  (R^2)^{-1} S^2 )^{\top} z^{(2)} + \int_E  (\mathbf{M}^{2}(e) -\mathbf{N}^2(e)  (R^2 )^{-1} S^2  )^{\top} \varrho(\mathrm de)  k^{(2)}(e)  \Big) \mathrm d s,   \quad s\geq t,      \\
&\!\! x^{(2)} (t)\!=\!\mathbb{E}[x_{t}].
\end{aligned}\right.
\end{equation}
%
%
Note that,  these two FBSDEs are coupled.  In fact,  \eqref{Hamil-x-2} is a deterministic system,  from which $(z^{(2)},k^{(2)}(\cdot))$   can not be worked out by itself.  Moreover, the solution  $(x^{(2)},y^{(2)})$  of \eqref{Hamil-x-2} appears in the  diffusion terms  of \eqref{Hamil-x-1}.
Therefore,  \eqref{Hamil} is  an infinite horizon mean-field fully coupled FBSDEs with jumps.

Now, based on  the Hamiltonian system  \eqref{Hamil}, we can give another characterization of the open-loop solvability of Problem (MF-LQ).

\begin{lemma}\label{Le-MFLQ-1}\sl
Assume  {\bf{(H$_{1}$)}}, {\bf{(H$_{2}$)}} and Condition {\bf (PD)} hold. Then the Hamiltonian system \eqref{Hamil} (or, \eqref{Hamil-x-1} and \eqref{Hamil-x-2}) has a unique  solution $(x(\cdot),y(\cdot),z(\cdot),k(\cdot,\cdot) )\in \mathcal{L}_{\mathbb{F}}^{2, K}(t, \infty)$ with $K<\kappa $, if and only if
Problem (MF-LQ) is uniquely open-loop solvability  at $(t,x_t)\in [0,\infty)\times L_{\mathcal{F}_t}^2(\Omega;\mathbb{R}^n)$ with the open-loop optimal control $u^*(\cdot)$ being in \eqref{optimal-u}.
\end{lemma}

From above, we know  the wellposedness of  Hamiltonian system \eqref{Hamil} is  crucial to study the open-loop  solvability of Problem (MF-LQ). Therefore,  we convert to the   research of  the wellposedness of  Hamiltonian system \eqref{Hamil}, which will be carried out in two cases:
(i)  the cross terms $S(\cdot)$ and $\bar S(\cdot)$ are all being zero matrices; (ii) at least one  of  $S(\cdot)$ or $\bar S(\cdot)$ is nonzero matrix.

\subsection{The case  $S(\cdot ) $ and  $\bar S(\cdot ) $ all being zero matrices} 
\label{subsection-S-0}

When   $S (\cdot )$ and $ \bar S(\cdot ) $ are all being zero matrices, FBSDE \eqref{Hamil} is reduced to
%
\begin{equation}\label{FBSDE-S=0}
\left\{
\begin{aligned}
&\!\! \mathrm dx (s)\!=\! \sum_{ \iota=1}^2\Big(A^\iota  x^{(\iota)}    -B^\iota   (R^\iota )^{-1}\mathbf{\Lambda}^\iota [y^{(\iota)},z^{(\iota)},k^{(\iota)}] \Big) \, \mathrm ds \\
&    +\sum\limits_{i=1}^d\sum_{ \iota=1}^2\Big(C_i^\iota x^{(\iota)} \!-\!D_i^\iota(R^\iota )^{-1}\mathbf{\Lambda}^\iota[y^{(\iota)},z^{(\iota)},k^{(\iota)}]\Big)   \, \mathrm dW_i(s) \\
&  +\sum\limits_{j=1}^l\ds\int_E\sum_{ \iota=1}^2 \Big(M_j^\iota(e) x^{(\iota)}  -N_j^\iota(e) (R^\iota )^{-1}\mathbf{\Lambda}^\iota[y^{(\iota)},z^{(\iota)},k^{(\iota)}]\Big)\,  \widetilde{\mu}_j(\mathrm ds, \mathrm de), \ s \geq t, \\
&\!\!  \mathrm dy (s)\! =\!-\sum_{\iota=1}^2\Big( Q^\iota x^{(\iota)} +(2 K I+A^\iota )^{\top}y^{(\iota)}
  +(\mathbf{C}^\iota )^{\top} z^{(\iota)}+ \int_E  \mathbf{M}^{\iota}(e)^{\top} \varrho(\mathrm de)  k^{(\iota)} (e) \Big) \, \mathrm d s \\
&+ \sum\limits_{i=1}^d  z_i(s) \, \mathrm dW_i(s)\!+\! \sum\limits_{j=1}^l \!\int_E k_j(s, e)   \widetilde{\mu}_j(\mathrm ds, \mathrm de),    \quad s \geq t,    \\
&\!\!x (t)\!=\!x_{t}.
\end{aligned}\right.
\end{equation}
The method of continuity   (referring to \cite{HP-1995, Peng-Wu-1999, Wei-Yu-2021}, etc.) will be adopted here to study the wellposedness of \eqref{FBSDE-S=0}. For this, we introduce
 a family of  infinite horizon mean-field  FBSDEs with jumps parameterized by $\alpha\in[0,1]$ as follows,
\begin{equation}  \label{Hamil-alpha}
\left\{
\begin{aligned}
&\!\! \mathrm dx_{\alpha}(s) = \Big[\sum_{\iota=1}^2\big(\alpha A^\iota  x^{(\iota)}_{\alpha}
-B^\iota  (R^\iota) ^{-1}\mathbf{\Lambda}^\iota[y^{(\iota)}_\alpha,z^{(\iota)}_\alpha,k^{(\iota)}_\alpha]\big)
-(1-\alpha) \kappa_1 x_{\alpha}  +\varphi \Big] \, \mathrm ds \\
&\!  +\sum\limits_{i=1}^d\[ \sum_{\iota=1}^2 \big( \alpha C_i^\iota  x^{(\iota)}_{\alpha}
-D_i^\iota  (R^\iota) ^{-1}\mathbf{\Lambda}^\iota[y^{(\iota)}_\alpha,z^{(\iota)}_\alpha,k^{(\iota)}_\alpha]\big) +\psi_i \]\, \mathrm dW_i(s)  \\
& \!  +\sum\limits_{j=1}^l\ds\int_E\[ \sum_{\iota=1}^2 \big(\alpha M_j^{\iota} (e) x^{(\iota)}_{\alpha}
-N_j^\iota(e)   (R^\iota) ^{-1}\mathbf{\Lambda}^\iota[y^{(\iota)}_\alpha,z^{(\iota)}_\alpha,k^{(\iota)}_\alpha]\big) + \chi_j(e) \] \, \widetilde{\mu}_j(\mathrm ds, \mathrm de) , \  s\geq t,\\
&\!\! \mathrm dy_{\alpha}(s)   =-\[\sum_{\iota=1}^2\alpha \(Q^\iota x^{(\iota)}_{\alpha}+(2 K I+ A^\iota) ^{\top} y^{(\iota)}_{\alpha}
+(\mathbf{C}^\iota) ^{\top} z^{(\iota)}_{\alpha}
+\ds\int_E  \mathbf{M}^{\iota}(e)  ^{\top}\varrho(\mathrm de)k^{(\iota)}_{\alpha}(e)  \)\\
&\!
-(1-\alpha) \kappa_2 y_{\alpha}  +\phi \] \, \mathrm ds
+ \sum\limits_{i=1}^dz_{\alpha,i}(s)  \, \mathrm dW_i (s)  + \ds\sum\limits_{j=1}^l\int_E k_{\alpha,j}(s,e) \, \widetilde{\mu} _j(\mathrm ds, \mathrm de),  \quad  s \geq t, \\
&\!\! x_{\alpha}(t)=\xi,
\end{aligned}\right.
\end{equation}
where  $\xi \in L_{\mathcal{F}_{t}}^{2}\left(\Omega ; \mathbb{R}^{n}\right)$ and
$\zeta(\cdot):=(\phi(\cdot)^{\top}, \f(\cdot)^{\top}, \psi(\cdot)^{\top},\chi(\cdot,\cdot)^\top)^{\top}\in \cL_{\mathbb{F}}^{2, K}(t, \infty)$.

For any $\alpha\in[0,1]$, we firstly study the a priori estimate for FBSDEs \eqref{Hamil-alpha}.

\begin{lemma}\label{Le-FBSDE-1}\sl
Assume   {\bf{(H$_{1}$)}}, {\bf{(H$_{2}$)}}, Condition {\bf (PD)} hold. Let
  $K<\kappa$ and  $\theta_\alpha(\cdot), \bar{\th}_\alpha(\cdot)\in \cL_{\mathbb{F}}^{2, K}(t, \infty )$ are the solutions to FBSDEs \eqref{Hamil-alpha}, $\alpha\in[0,1]$ with $(\xi, \zeta (\cdot))$, $(\bar{\xi}, \bar{\zeta}(\cdot))\in L_{\cF_{t}}^{2}(\Omega, \mathbb{R}^{n}) \times \cL_{\mathbb{F}}^{2, K}(t, \infty)$, respectively.
Then, there exists some constant ${\mathbf{k}}>0$, such that for any $\displaystyle K\in \[\frac{\kappa_1-\kappa_2}{2}, \frac{\kappa_1-\kappa_2}{2}+{\mathbf{k}}\)$ and $\alpha \in[0,1]$,
\begin{equation}\label{equ-FBSDE-8}
\begin{aligned}
&\mathbb{E}\[\big|(y_\alpha(t)-\bar{y}_\alpha(t)) \mathrm e^{K t}\big|^{2}
+\ds\int_{t}^{\infty} \big|(\theta_\alpha(s)-\bar{\theta}_\alpha(s))\mathrm e^{K s}\big|^{2}  \mathrm ds\] \\
&\leq C\mathbb{E}\[\big|(\xi-\bar{\xi})\mathrm e^{K t}\big|^{2}
+\ds\int_{t}^{\infty} \big|(\zeta(s)-\bar{\zeta}(s))\mathrm e^{K s}\big|^{2}   \mathrm ds\],
\end{aligned}
\end{equation}
where the  constant $C>0$ depends on $\kappa$, $L_{\varepsilon,j}$, $j=1,2,3$
and the norms of the coefficients of \eqref{FBSDE-S=0}.
%
\end{lemma}
\begin{proof}
For convenience, we denote
$
\widehat{\xi}=(\xi-\bar{\xi}) \mathrm e^{K t}$  and $  \widehat{\vartheta}(\cdot)=(\vartheta(\cdot)-\bar{\vartheta}(\cdot)) e^{K \cdot},
$
where $\vartheta(\cdot)  =x_\alpha(\cdot),$ $y_\alpha(\cdot),$ $z_\alpha(\cdot),$ $k_\alpha(\cdot, \cdot),$ $\zeta(\cdot).$

We denote $\kappa_1'$, $\kappa_2'$ and $\kappa'$ by the ones defined  in  \eqref{kappa} corresponding to  the forward SDE in \eqref{Hamil-alpha}, $\alpha\in[0,1]$.
By the  direct  calculations, we have $\kappa'_1=\kappa_1,$ $\kappa'_2 \geq \kappa,$ $\kappa'\geq \kappa.$ Then,   applying  \eqref{equ-SDE-4} in Lemma \ref{Le-SDE-1} to $x_\alpha(\cdot)$ and $\bar{x}_\alpha(\cdot)$, we get
\begin{equation}\label{equ-FBSDE-9}
 \begin{aligned}
&(- 2 K+ 2 \kappa  - 3\varepsilon) \mathbb{E}\int_t^\infty |\widehat{x}_\alpha|^{2}  \, \mathrm  ds
\leq(- 2 K+ 2 \kappa'  - 3\varepsilon) \mathbb{E}\int_t^\infty |\widehat{x}_\alpha|^{2}  \, \mathrm  ds \\
&\leq L_{\varepsilon,1} \mathbb{E}| \widehat{\xi} |^{2}+ \mathbb{E}\int_t^\infty \(\frac {L_{\varepsilon,1}}{\varepsilon}
\Big|\widehat{\varphi} -\sum_{\iota=1}^2 B^\iota (R^\iota)^{-1} \mathbf{\Lambda}^\iota[\widehat y^{(\iota)}_\alpha,\widehat z^{(\iota)}_\alpha,\widehat k^{(\iota)}_\alpha]
 \Big|^{2} \\
&\quad
+\(2+\frac{\|\mathbf{C}^{1}(s)\|^{2}}{\varepsilon}\) \Big|  \widehat{\psi} -\sum_{\iota=1}^2 \mathbf{D}^\iota (R^\iota)^{-1}\mathbf{\Lambda}^\iota[\widehat y^{(\iota)}_\alpha,\widehat z^{(\iota)}_\alpha,\widehat k^{(\iota)}_\alpha]
 \Big|^{2}\\
&\quad
+\(2+\frac{\|\mathbf{M}^{1}(s, \cdot)\|_\rho^{2}}{\varepsilon}\) \Big\| \widehat{\chi}(\cdot) - \sum_{\iota=1}^2 \mathbf{N}^\iota(\cdot) (R^\iota)^{-1} \mathbf{\Lambda}^\iota[\widehat y^{(\iota)}_\alpha,\widehat z^{(\iota)}_\alpha,\widehat k^{(\iota)}_\alpha] \Big\|_\rho^{2} \)\,\mathrm  ds \\
&\leq L_{\varepsilon,1} \mathbb{E}| \widehat{\xi} |^{2} + \mathbb{E}\int_t^\infty \( L_{\varepsilon,4}\sum_{\iota=1}^2 \big| \mathbf{\Lambda}^\iota[\widehat y^{(\iota)}_\alpha,\widehat z^{(\iota)}_\alpha,\widehat k^{(\iota)}_\alpha]\big|^{2} \\
&\quad
+\frac {3 L_{\varepsilon,1}}{\varepsilon}|\widehat{\varphi} |^{2} +3 \(2+\frac{\|\mathbf{C}^{1}(s)\|^{2}}{\varepsilon}\)|\widehat{\psi} |^{2} +3\(2+\frac{\|\mathbf{M}^{1}(s, \cdot)\|_\rho^{2}}{\varepsilon}\)\|\widehat{\chi}(s, \cdot)\|_\rho^{2}\)  \,\mathrm  ds,
\end{aligned}
\end{equation}
where $L_{\varepsilon,1}$ is the one in Lemma \ref{Le-SDE-1}, $L_{\varepsilon,4}$ is a positive constant depending on $\varepsilon$, $L_{\varepsilon,1}$, $\sup\limits_{s\in[0,\i)}\|\mathbf{C}^{1}(s)\|^{2}$, $\sup\limits_{s\in[0,\i)}\|\mathbf{M}^{1}(s, \cdot)\|_\rho^{2}$, $ \sup\limits_{s\in[0,\i)}\|B^\iota(s) R^\iota(s)^{-1}\|^{2}$, $\sup\limits_{s\in[0,\i)}\|\mathbf{D}^\iota(s) R^\iota(s)^{-1}\|^{2}$ and $\sup\limits_{s\in[0,\i)}\|\mathbf{N}^\iota(s, \cdot) R^\iota(s)^{-1}\|_\rho^{2}$, $\iota=1,2$.

Setting $\e_{1}:=\frac{-2K+2\kappa }{3}> 0,$  when $\varepsilon \in\left(0, \e_{1}\right)$, we get
\begin{equation}\label{equ-FBSDE-14}
\mathbb{E}\ds\int_{t}^{\infty}|\widehat{x}_\alpha|^{2}\mathrm ds
\leq   C_{1,\varepsilon} \mathbb{E}\[|\widehat{\xi}|^{2}+\ds\int_{t}^{\infty}\(\sum_{\iota=1}^2 \big|\mathbf{\Lambda}^\iota[\widehat y^{(\iota)}_\alpha,\widehat z^{(\iota)}_\alpha,\widehat k^{(\iota)}_\alpha]\big|^{2}  + |\widehat{\varphi}|^{2}+ |\widehat{\psi}|^{2}+ \|\widehat{\chi}(\cdot)\|_\rho^{2}\)\,\mathrm ds\],
\end{equation}
with
\begin{equation}\label{equ-FBSDE-13}
C_{1,\varepsilon} :=\frac{1}{-2K+ 2\kappa - 3\varepsilon}\max\Big\{\! L_{\varepsilon,1} ,L_{\varepsilon,4},\frac {3 L_{\varepsilon,1}}{\varepsilon},3 \(2+\frac{\sup\limits_{s\in[0,\i)}\|\mathbf{C}^{1}(s)\|^{2}\vee \sup\limits_{s\in[0,\i)}\|\mathbf{M}^{1}(s, \cdot)\|_\rho^{2}}{\varepsilon}\) \!  \Big\}.\!\!\!\!
\end{equation}

Similarly,  denoting $\kappa_1''$, $\kappa_2''$ and $\kappa''$ by the ones  in  \eqref{kappa} corresponding to  the BSDE in \eqref{Hamil-alpha}, $\alpha\in[0,1]$
 and by the  direct  calculations, we have $\kappa''_1\geq\kappa,$ $ \kappa''_2 \geq  \kappa_2,$ $\kappa''\geq \kappa.$
Then, by applying  \eqref{equ-BSDE-4-1}  in Lemma \ref{Le-BSDE-1} to $(y_\alpha(\cdot), z_\alpha(\cdot), k_\alpha(\cdot, \cd))$ and $(\bar{y}_\alpha(\cdot), \bar{z}_\alpha(\cdot),  \bar{k}_\alpha(\cdot, \cd))$ yields, we get
\begin{equation}\label{equ-FBSDE-15}
\begin{aligned}
&\mathbb{E} |\widehat{y}_\alpha(t)|^{2}+\mathbb{E} \int_{t}^{\infty}\( (2 \kappa-2 K-\varepsilon)|\widehat{y}_\alpha|^{2}
+|\widehat{z}_\alpha|^{2} +\|\widehat{k}_\alpha(\cdot)\|_\rho^{2}\) \mathrm ds  \\
&\leq \mathbb{E} |\widehat{y}_\alpha(t)|^{2}+\mathbb{E} \int_{t}^{\infty}\( (2 \kappa''-2 \alpha K-\varepsilon) |\widehat{y}_\alpha|^{2}
+|\widehat{z}_\alpha|^{2} +\|\widehat{k}_\alpha(\cdot)\|_\rho^{2}\) \mathrm ds  \\
&\leq\(\frac{1}{\varepsilon}+ L_{\varepsilon,2} +  L_{\varepsilon,3} \)
\mathbb{E}\int_{t}^{\infty} \Big|\sum_{\iota=1}^2 Q^\iota\widehat{x}_\alpha^{\iota}
+\widehat{\phi}\Big|^{2} \mathrm ds \\
&\leq 3\(\frac{1}{\varepsilon}+  L_{\varepsilon,2} +  L_{\varepsilon,3} \)
\max\Big\{ 1, \sup\limits_{s\in[0,\i)}\|Q^1(s)\|^{2}, \sup\limits_{s\in[0,\i)}\|Q^2(s)\|^{2}\Big\}
\mathbb{E}\int_{t}^{\infty} (|\widehat{x}_\alpha|^{2} +|\widehat{\phi}|^{2} )\,\mathrm ds ,
\end{aligned}
\end{equation}
where $L_{\varepsilon ,2}$ and $L_{\varepsilon,3}$ are the ones in Lemma \ref{Le-BSDE-1}.
Setting $\e_{2}:=-2K+2\kappa > 0,$ then   $2 \kappa_2 - 2 K >0$. So, when $\varepsilon \in(0, \e_{2})$,
\begin{equation}\label{equ-FBSDE-16}
 \mathbb{E}\[|\widehat{y}_\alpha(t)|^{2} + \int_{t}^{\infty}( |\widehat{y}_\alpha|^{2}
+|\widehat{z}_\alpha|^{2} +\|\widehat{k}_\alpha(\cdot)\|_\rho^{2}) \mathrm ds\]
\leq C_{2,\varepsilon} \mathbb{E} \int_{t}^{\infty}(|\widehat{x}_\alpha|^{2} + |\widehat{\phi}|^{2}) \mathrm ds,
\end{equation}
with
$$
  C_{2,\varepsilon} :=
\frac{3 (\frac{1}{\varepsilon}+ L_{\varepsilon,2} + L_{\varepsilon,3} )
\max\{ 1, \sup\limits_{s\in[0,\i)}\|Q^1(s)\|^{2}, \sup\limits_{s\in[0,\i)}\|Q^2(s)\|^{2} \}}{\min \{1, 2\kappa-2K-\varepsilon \}}.
$$

Further, for any $T >t$,  using It\^o's formula to $\langle\widehat{x}_\alpha (\cdot), \widehat{y}_\alpha (\cdot)\rangle$ on the interval $[t, T]$, we have
$$
\begin{aligned}
&\mathbb{E}\big[\langle\widehat{x}_\alpha (T), \widehat{y}_\alpha (T)\rangle-\langle\widehat{\xi}, \widehat{y}_\alpha (t)\rangle\big]
%
%
 \leq\mathbb{E }\ds\int_{t}^{T}\!\!\( \!\!-\sum_{\iota=1}^2\big\langle (R^i)^{-1}\mathbf{\Lambda}^\iota[\widehat y^{(\iota)}_\alpha,\widehat z^{(\iota)}_\alpha,\widehat k^{(\iota)}_\alpha],
\mathbf{\Lambda}^\iota[\widehat y^{(\iota)}_\alpha,\widehat z^{(\iota)}_\alpha,\widehat k^{(\iota)}_\alpha]\big\rangle\\
&+(1-\a) [2K-(\k_1-\k_2)] \langle\widehat{x}_\alpha , \widehat{y}_\alpha \rangle +\langle\widehat{\varphi}, \widehat{y}_\alpha \rangle - \langle\widehat{\phi}, \widehat{x}_\alpha \rangle
+\langle\widehat{\psi}, \widehat{z}_\alpha \rangle
+\langle\widehat{\chi}(\cdot), \widehat{k}_\alpha (\cdot)\rangle_\rho\)\mathrm ds .
\end{aligned}
$$
By Condition {\bf (PD)} and letting $T\to\i$,  we get
\begin{equation}\label{equ-FBSDE-17}
\begin{aligned}
&\mathbb{E} \ds\int_{t}^{\infty}\!\!\sum_{i=1}^2\big|\mathbf{\Lambda}^\iota[\widehat y^{(\iota)}_\alpha,\widehat z^{(\iota)}_\alpha,\widehat k^{(\iota)}_\alpha] \big|^{2} \mathrm ds \leq \frac{1}{L_{R,\widetilde {R}}}\mathbb{E}\[\langle\widehat{\xi}, \widehat{y}_\alpha(t)\rangle
+\!\!\ds\int_{t}^{\infty}\big((1-\a)[2K-(\k_1-\k_2)]\langle\widehat{x}_\alpha, \widehat{y}_\alpha \rangle\\
&\hskip6cm+\langle\widehat{\varphi}, \widehat{y}_{\alpha}\rangle - \langle\widehat{\phi}, \widehat{x}_{\alpha}\rangle+\langle\widehat{\psi}, \widehat{z}_{\alpha}\rangle
+\langle\widehat{\chi}(\cdot), \widehat{k}_{\alpha}(\cdot)\rangle_\rho
 \big)\,\mathrm ds \],
\end{aligned}
\end{equation}
where $ L_{R,\widetilde {R}}:= \frac{1}{\min\big\{ \mathop{\inf}\limits_{s \in[0, \infty)}\lambda_{\min }(R^1 (s)^{-1}), \mathop{\inf}\limits_{s \in[0, \infty)}\lambda_{\min }(R^2(s)^{-1})\big\}} $.

Restricting $0<\e<\min \left\{\e_{1}, \e_{2}\right\}$ and substituting \eqref{equ-FBSDE-14} into \eqref{equ-FBSDE-16} yields,
$$
\begin{aligned}
&\mathbb{E}\[|\widehat{y}_\alpha (t)|^{2} + \ds\int_{t}^{\infty}\big(|\widehat{y}_\alpha |^{2}+|\widehat{z}_\alpha |^{2}
+\|\widehat{k}_\alpha (\cdot)\|_\rho^{2}\big) \mathrm ds\]\leq  C_{2,\varepsilon} \mathbb{E} \ds\int_{t}^{\infty} |\widehat{\phi}|^{2} \mathrm ds \\
&\ds + C_{2,\varepsilon} C_{1,\varepsilon} \mathbb{E}\[ |\widehat{\xi}|^{2}
+\ds\int_{t}^{\infty}\(\sum_{\iota=1}^2 \big|\mathbf{\Lambda}^\iota[\widehat y^{(\iota)}_\alpha,\widehat z^{(\iota)}_\alpha,\widehat k^{(\iota)}_\alpha] \big|^{2}
+ |\widehat{\f}|^{2} + |\widehat{\psi}|^{2} + \|\widehat{\chi}(\cdot)\|_\rho^{2}\) \,\mathrm ds
\].
\end{aligned}
$$
Then, combing  \eqref{equ-FBSDE-14},
\begin{equation}\label{equ-FBSDE-18}
\ds\mathbb{E}\[|\widehat{y}_\alpha (t)|^{2} + \int_{t}^{\infty}\!\!|\widehat{\th}_\alpha|^{2}\mathrm ds\]
\leq C_{3}\mathbb{E}\[|\widehat{\xi}|^{2}
+\!\int_{t}^{\infty}\!\! \(\sum_{\iota=1}^2 \big|\mathbf{\Lambda}^\iota[\widehat y^{(\iota)}_\alpha,\widehat z^{(\iota)}_\alpha,\widehat k^{(\iota)}_\alpha] \big|^{2} +|\widehat{\zeta}|^{2}\) \mathrm ds\],\!
\end{equation}
where  $
C_{3,\varepsilon}:=\max \left\{C_{2,\varepsilon}, C_{1,\varepsilon}\left(1+C_{2,\varepsilon}\right)\right\}.
$

Further, substituting \eqref{equ-FBSDE-17} into \eqref{equ-FBSDE-18}, if $K\geq \frac{\kappa_1-\kappa_2}{2}$,
\begin{equation}\label{equ-FBSDE-20}
\begin{aligned}
&\mathbb{E}\[|\widehat{y}_\alpha (t)|^{2} + \ds\int_{t}^{\infty}|\widehat{\th}_\alpha|^{2} \mathrm ds \] \\
&\leq C_{3,\varepsilon}\mathbb{E}\[|\widehat{\xi}|^{2}\!+\frac{\langle\widehat{\xi}, \widehat{y}_\alpha (t)\rangle
 }{L_{R,\widetilde {R}}}
+\frac{1}{L_{R,\widetilde {R}}}\!\ds\int_{t}^{\infty}\!\( (1-\a)[2K-(\k_1-\k_2)]\langle\widehat{x}_\alpha, \widehat{y}_\alpha \rangle +|\widehat{\zeta}|^{2} \\
&\qquad
 +\langle\widehat{\varphi}, \widehat{y}_\alpha \rangle
- \langle\widehat{\phi}, \widehat{x}_\alpha \rangle+\langle\widehat{\psi}, \widehat{z}_\alpha \rangle
+\langle\widehat{\chi}(\cdot), \widehat{k}_\alpha (\cdot)\rangle_\rho  \)\mathrm ds \]\!\!\!\!\\
&\leq C_{3} \mathbb{E} \[ \frac{\epsilon}{  C_3,\varepsilon} |\widehat{y}_\alpha (t)|^{2}
+(1+\frac{C_3,\varepsilon}{4\epsilon L_{R,\widetilde {R}}^{2}} )|\widehat{\xi}|^{2}\\
 &\qquad + \ds\int_{t}^{\infty}\(\frac{\epsilon}{ C_{3,\varepsilon}} |\widehat{\theta}_\alpha|^{2}
+(1+\frac{ C_{3,\varepsilon}}{ 4\epsilon L_{R,\widetilde {R}}^{2}} )|\widehat{\zeta}|^{2}
+ \frac{2K-(\k_1-\k_2)}{\epsilon L_{R,\widetilde {R}}  } (|\widehat{x}_\alpha|^{2}+|\widehat{y}_\alpha|^{2} ) \) \mathrm ds\] \\
&\leq\( \frac{C_{3,\varepsilon}(2K-(\k_1-\k_2))}{\epsilon L_{R,\widetilde {R}}  } +\epsilon\)\mathbb{E}\[|\widehat{y}_\alpha (t)|^{2} + \ds\int_{t}^{\infty}|\widehat{\th}_\alpha|^{2} \mathrm ds \]
+\dfrac{C_{4,\varepsilon}}{2} \mathbb{E}\[|\widehat{\xi}|^{2} + \ds\int_{t}^{\infty}|\widehat{\zeta}|^{2} \mathrm ds \],
\end{aligned}
\end{equation}
where $\epsilon\in (0,1)$,
$
C_{4,\varepsilon}:=   C_{3,\varepsilon}\(1+\frac{C_{3,\varepsilon}}{4\epsilon L_{R,\widetilde {R}}^{2}} \).
$
Then, for any $\epsilon\in(0,1)$, provided $K< \frac{\kappa_1-\kappa_2}{2}+{\mathbf{k}} $ with ${\mathbf{k}}=\frac{L_{R,\widetilde {R}} \epsilon(1-\epsilon) }{2C_{3,\varepsilon}  } $,
  we  get \eqref{equ-FBSDE-8}.

\end{proof}

Then, the continuation lemma can be established  based on the above a priori estimate.
\begin{lemma}\label{Le-FBSDE-2}\sl
Under the same assumptions   of Lemma \ref{Le-FBSDE-1}.
Then, we get the existence of some constant $\d_{0}>0$ independent of $\a$ such that if for some $\a_{0} \in[0, 1)$, FBSDEs \eqref{Hamil-alpha}, $\alpha_0\in[0,1]$ admits a unique solution in $\cL_{\mathbb{F}}^{2, K}\left(t, \infty \right)$ for any
$(\xi, \zeta(\cdot)) \in L_{\mathcal{F}_{t}}^{2}\left(\Omega ; \mathbb{R}^{n}\right) \times \cL_{\mathbb{F}}^{2, K}\left(t, \infty \right)$,
then the same   is true for \eqref{Hamil-alpha}  with $\a = \a_{0} + \d$ and $\d \in\left[0, \d_{0}\right]$, $\a \leq 1$.
\end{lemma}
\begin{proof}
Let $\d_{0}>0$ be some constant which will be determined later.
For any  $\d \in\left[0, \d_{0}\right]$, $\xi \in L_{\mathcal{F}_{t}}^{2}\left(\Omega ; \mathbb{R}^{n}\right)$, $\zeta(\cdot) \in \cL_{\mathbb{F}}^{2, K}\big(t, \infty \big)$ and $\theta(\cdot) \in \cL_{\mathbb{F}}^{2,K}(t, \infty)$,
we consider the following infinite horizon mean-field  FBSDEs with jumps,
\begin{equation}\label{equ-FBSDE-23}
 \left\{
\begin{aligned}
&\!\! \mathrm dX(s)\! =\!\[ \sum_{\iota=1}^2\big(\alpha_0 A^\iota  X^{(\iota)}
-B^\iota  (R^\iota )^{-1}\mathbf{\Lambda}^\iota[Y^{(\iota)},Z^{(\iota)},K^{(\iota)}]\big) -(1-\alpha_0) \kappa_1 X  + \varphi_1 \]   \mathrm ds \\
& \! +\sum\limits_{i=1}^d\[\sum_{\iota=1}^2 \big( \alpha_0 C_i^\iota  X^{(\iota)}
-D_i^\iota  (R^\iota) ^{-1}\mathbf{\Lambda}^\iota[Y^{(\iota)},Z^{(\iota)},K^{(\iota)}]\big) +\psi_{1,i} \]\, \mathrm dW_i(s)  \\
& \!    +\sum\limits_{j=1}^l\ds\int_E\[ \sum_{\iota=1}^2 \!\big(\alpha_0 M_j^{\iota} (e)  X^{(\iota)}
-N_j^\iota(e)   (R^\iota) ^{-1}\mathbf{\Lambda}^\iota[Y^{(\iota)},Z^{(\iota)},K^{(\iota)}]\big) + \chi_{1,j}(e) \] \, \widetilde{\mu}_j(\mathrm ds, \mathrm de) ,\! \  s\geq t,\!\!\!\!\!\\
&\!\! \mathrm dY(s)  \! =\!-\[\sum_{\iota=1}^2 \alpha_0 \big[(2 K I+A^\iota) ^{\top}Y^{(\iota)}+(\mathbf{C}^\iota )^{\top} Z^{(\iota)}
+\ds\int_E  \mathbf{M}^{i}(e)^{\top}\varrho(\mathrm de)K^{(\iota)}(e)+Q^\iota X^{(\iota)} \big]\\
&
-(1-\alpha_0) \kappa_2 Y +\phi_1 \]  \mathrm ds
+  \sum\limits_{i=1}^dZ_i(s)  \, \mathrm dW_i (s)  + \ds\sum\limits_{j=1}^l\int_E K_j(s,e) \, \widetilde{\mu} _j(\mathrm ds, \mathrm de),  \quad  s \geq t, \\
&\!\! X(t)\!=\!\xi,
\end{aligned}\right.\!
\end{equation}
where $$\begin{aligned}
&\varphi_{1}(\cdot):=\delta \sum_{\iota=1}^2 A^\iota(\cdot) x^{(\iota)}(\cdot)+\delta \kappa_1 x(\cdot)+ \varphi(\cdot),\q \psi_{1,i}(\cdot):= \delta \sum_{\iota=1}^2 C_i^\iota(\cdot) x^{(\iota)}(\cdot)+ \psi_i(\cdot).\\
&\chi_{1,j}(\cdot,\cdot):=\delta \sum_{\iota=1}^2  M_j^{\iota}(\cdot,\cdot) x^{(\iota)}(\cdot)+ \chi_j(\cdot,\cdot),\\
&\phi_1(\cdot):=\delta \sum_{\iota=1}^2 \big[(2 K I+A^\iota(\cdot))^{\top} y^{(\iota)}(\cdot)
+\mathbf{C}^\iota(\cdot)^{\top} z^{(\iota)}(\cdot) + \int_E \mathbf{M}^{\iota}(\cdot,e)^{\top}\varrho(\mathrm de)k^{(\iota)}(\cdot,e)+Q^\iota(\cdot) x^{(\iota)}(\cdot)\big] \\
&\qquad\quad +\delta \kappa_2 y(\cdot)+\phi(\cdot).
\end{aligned}$$
In fact,  the above  FBSDEs \eqref{equ-FBSDE-23} defines a mapping $\mathcal{T}_{\a_{0} + \d} $ as
$$
 \Theta(\cdot):=\big(X(\cdot)^{\top}, Y(\cdot)^{\top}, Z(\cdot)^{\top}, K(\cdot, \cd)^{\top}\big)^{\top}=\mathcal{T}_{\a_{0} + \d} (\theta(\cdot)):\cL_{\mathbb{F}}^{2, K}\left(t, \infty \right) \mapsto\cL_{\mathbb{F}}^{2, K}\left(t, \infty \right).
$$
which is guaranteed by the given unique solvability of   FBSDEs \eqref{equ-FBSDE-23}   in   $\cL_{\mathbb{F}}^{2, K}\left(t, \infty \right)$
and the arbitrariness of $\theta(\cdot)$.
We claim  that the mapping $\mathcal{T}_{\a_{0} + \d} $  is contractive when $\d$ is small enough.
For this, consider any  $\theta(\cdot)=(x(\cdot)^{\top},y(\cdot)^{\top},z(\cdot)^{\top},k(\cdot,\cdot)^{\top})^{\top},$ $ \bar{\th}(\cdot) =(\bar x(\cdot)^{\top},\bar y(\cdot)^{\top},\bar z(\cdot)^{\top},\bar k(\cdot,\cdot)^{\top})^{\top}\in \cL_{\mathbb{F}}^{2, K}\left(t, \infty \right)$,
let $\Theta(\cdot) =(X(\cdot)^{\top},Y(\cdot)^{\top},Z(\cdot)^{\top},K(\cdot,\cdot)^{\top})^{\top}=\mathcal{T}_{\a_{0}+\d}(\theta(\cdot))$,
 $\bar{\Theta}(\cdot) =(\bar X(\cdot)^{\top},\bar Y(\cdot)^{\top},\bar Z(\cdot)^{\top},\bar K(\cdot,\cdot)^{\top})^{\top}=\mathcal{T}_{\a_{0}+\d}(\bar{\th}(\cdot))$.
Setting $ \widehat \theta(\cdot):=\theta(\cdot)-\bar \theta(\cdot)$, the a prior estimate in Lemma \ref{Le-FBSDE-1} implies
$$
\begin{aligned}
&\mathbb{E}\[| {Y}(t)-\bar{Y}(t)|^{2} + \ds\int_{t}^{\infty}| {\Th}(s)-\bar{\Th}(s)|^{2}\mathrm ds\]\\
&\leq C \d^{2} \mathbb{E} \ds\int_{t}^{\infty}\(\big|\sum_{\iota=1}^2 A^\iota \widehat{x}^{(\iota)} + \kappa_1\widehat{x} \big|^{2}
+\sum\limits_{i=1}^d \big|\sum_{\iota=1}^2 C_i^\iota \widehat{x}^{(\iota)}\big|^{2}
+\sum\limits_{j=1}^l\int_E\big|\sum_{\iota=1}^2 M_j^{\iota}(e) \widehat{x}^{(\iota)} \big|^2\rho_j(\mathrm de)\\
&\quad + \big|\sum_{\iota=1}^2 \big[(2 K I+A^\iota)^{\top}\widehat{y}^{(\iota)}
+(\mathbf{C}^\iota)^{\top} \widehat{z}^{(\iota)} +\ds\int_E \mathbf{M}^\iota(e)^{\top} \varrho(\mathrm de) \widehat{k}^{(\iota)}(e) +Q^\iota\widehat{x}^{(\iota)} \big] + \kappa_2 \widehat y \big|^{2}
\)\mathrm ds\\
& \leq C \d^{2} \mathbb{E} \ds\int_{t}^{\infty}| {\theta}(s)-\bar \theta(s)|^{2}\mathrm ds,
\end{aligned}
$$
where the constant $C>0$ is independent of $\a_{0}$ and $\d$.
Then, the mapping $\mathcal{T}_{\a_{0}+\d}$ is contractive if we choose
$\d \in[0, \d_{0}]$ with $\d_{0}=\frac{1}{2 \sqrt{C}}$.
Therefore, we get the existence of  the unique fixed point   $\Th^{*}(\cdot) \in \cL_{\mathbb{F}}^{2, K}\left(t, \infty \right)$, which  is indeed the unique solution to MF-FBSDEs \eqref{Hamil-alpha} with $\alpha=\alpha_0+\delta $.

\end{proof}
\begin{lemma}\label{Le-FBSDE-0}\sl
Assume that {\bf{(H$_{1}$)}}  holds and $\k_1> -\k_2$. Then, for any
$K \in \left(-\k_2, \k_1\right)$, any $\xi \in L_{\mathcal{F}_{t}}^{2}\left(\Omega; \mathbb{R}^{n}\right)$, and
$\zeta(\cdot)\in \cL_{\mathbb{F}}^{2, K}(t, \infty)$,
 MF-FBSDEs \eqref{Hamil-alpha} with $\alpha=0$ admits a unique solution in the space $\cL_{\mathbb{F}}^{2, K}(t, \infty).$
\end{lemma}
\begin{proof}
 Obviously, FBSDEs \eqref{Hamil-alpha}$_0$ reads
\begin{equation}\label{equ-FBSDE-7}
\left\{
\begin{aligned}
&\!\! \mathrm dx_0(s) = \[-\sum_{\iota=1}^2 B^\iota  (R^\iota) ^{-1}\mathbf{\Lambda}^\iota[y^{(\iota)}_0,z^{(\iota)}_0,k^{(\iota)}_0]
-\kappa_1 x_0  +\varphi \]  \, \mathrm ds \\
& \quad +\sum\limits_{i=1}^d\[-\sum_{\iota=1}^2 D_i^\iota  (R^\iota) ^{-1}\mathbf{\Lambda}^\iota[y^{(\iota)}_0,z^{(\iota)}_0,k^{(\iota)}_0] +\psi_i \]\, \mathrm dW_i(s)  \\
& \quad +\sum\limits_{j=1}^l\ds\int_E\[- \sum_{\iota=1}^2N_j^\iota(e) (R^\iota) ^{-1}\mathbf{\Lambda}^\iota[y^{(\iota)}_0,z^{(\iota)}_0,k^{(\iota)}_0]  +\chi_j(e) \] \, \widetilde{\mu}_j(\mathrm ds, \mathrm de), \quad s\geq t,\\
&\!\! \mathrm dy_0(s)   =[\kappa_2 y_0 -  \phi   ] \, \mathrm ds +\sum\limits_{i=1}^dz_{0,i}(s)  \, \mathrm dW_i (s)  + \ds\sum\limits_{j=1}^l\int_E k_{0,j}(s,e) \, \widetilde{\mu} _j(\mathrm ds, \mathrm de),  \quad  s \geq t, \\
&\!\! x_0(t)=\xi,
\end{aligned}\right.
\end{equation}
 which is in fact  decoupled.

The BSDE in \eqref{equ-FBSDE-7}  is an infinite horizon BSDE without mean-field terms, so that we  can apply Lemma 2.5 in \cite{Wei-Yu-2021} to get $(y_{0}(\cdot),z_{0}(\cdot), k_{0}(\cdot, \cd))\in L_{\mathbb{F}}^{2, K} (t, \infty ; \mathbb{R}^{n})
\times L_{\mathbb{F}}^{2, K}(t, \infty ; \mathbb{R}^{n d})
\times \cK_{\mathbb{F}}^{2, K}(t, \infty ; \mathbb{R}^{n l  })$  firstly when $K>-\kappa_2$.
Further, according to  Lemma 2.2 in  \cite{Wei-Yu-2021}, when $K<\kappa_1$, $x_{0}(\cdot)\in L_{\mathbb{F}}^{2, K} (t, \infty ; \mathbb{R}^{n})$ can be solved from the SDE in \eqref{equ-FBSDE-7} under the known $(y_{0}(\cdot),z_{0}(\cdot), k_{0}(\cdot, \cd))$.
Therefore, the desired result is got.
 \end{proof}
%
%
%
%
%

%

%
%
\begin{proposition}\label{Th-FBSDE-1}\sl
Assume {\bf{(H$_{1}$)}}, {\bf{(H$_{2}$)}} and Condition {\bf (PD)} hold.
Let   $K<\kappa$.
Then,  there exists some constant ${\mathbf{k}}>0$, such that for any $\displaystyle K\in \[\frac{\kappa_1-\kappa_2}{2}, \frac{\kappa_1-\kappa_2}{2}+{\mathbf{k}}\)$, for any $\left(t, x_{t}\right) \in$ $[0, \infty) \times L_{\mathcal{F}_{t}}^{2}\left(\Omega ; \mathbb{R}^{n}\right)$,
FBSDEs \eqref{FBSDE-S=0} admits a unique solution $\theta (\cdot) \in \cL_{\mathbb{F}}^{2, K}(t, \infty)$.
Moreover,
\begin{equation}\label{equ-FBSDE-4}
\begin{aligned}
&\mathbb{E}\[|y (t)\mathrm e^{K t}|^{2} + \ds\int_{t}^{\infty}|\theta (s) \mathrm e^{K s}|^{2} \mathrm ds\]
\leq C \mathbb{E}|x_t \mathrm e^{K t}|^{2},
\end{aligned}
\end{equation}
where $C$ is the same constant as in \eqref{equ-FBSDE-8}.
Furthermore, let $\bar{\th}(\cdot) \in \cL_{\mathbb{F}}^{2, K}(t, \infty)$ be a solution to  FBSDEs \eqref{FBSDE-S=0}
with another $\bar{x}_t\in L_{\mathcal{F}_{t}}^{2}\left(\Omega ; \mathbb{R}^{n}\right)$. Then,
\begin{equation}\label{equ-FBSDE-5}
\begin{aligned}
&\mathbb{E}\[\big|(y (t)-\bar{y} (t)) \mathrm e^{K t}\big|^{2}
+\ds\int_{t}^{\infty}\big|(\theta (s)-\bar{\theta} (s)) \mathrm e^{K s}\big|^{2}\mathrm ds \]
 \leq C \mathbb{E}|(x_t-\bar{x}_t) \mathrm e^{ K t}|^{2}.
\end{aligned}
\end{equation}
\end{proposition}
\begin{proof}
Note that $K=\frac {\kappa_1-\kappa_2}{2}$ and $K<\kappa$ imply $\k_1> -\k_2$.
Therefore, for any $(\xi, \zeta(\cdot))\in  L_{\mathcal{F}_{t}}^{2}\left(\Omega; \mathbb{R}^{n}\right)\times\cL_{\mathbb{F}}^{2, K}(t, \infty)$  and  $\a \in[0, 1]$, we can apply  Lemma  \ref{Le-FBSDE-0} and  Lemma \ref{Le-FBSDE-2} to get the uniquely solvability of    FBSDEs \eqref{Hamil-alpha}, $\alpha\in[0,1]$  in $\cL_{\mathbb{F}}^{2, K}\left(t, \infty \right)$.
Specially, when $\a=1$ and $(\xi, \zeta(\cdot))=(x_t, 0)$, \eqref{Hamil-alpha}  coincides with \eqref{FBSDE-S=0}. Therefore,  the unique solvability of \eqref{FBSDE-S=0} is derived.

Finally, the estimate \eqref{equ-FBSDE-5} follows from \eqref{equ-FBSDE-8} by letting $\a=1, (\xi, \zeta(\cdot))=(x_t, 0)$ and $(\bar{\xi}, \bar{\zeta}(\cdot))=(\bar{x}_t, 0)$. Moreover,  if $(\bar{\xi}, \bar{\zeta}(\cdot))=(0, 0)$, we get \eqref{equ-FBSDE-4}    from \eqref{equ-FBSDE-5}.
%

%
\end{proof}
\subsection{The case  $S(\cd )$, $\bar S(\cd )$ being not all zero matrices} 
\label{subsection-S-neq-0}
Now, we go back to the general case \eqref{Hamil}, i.e., $S(\cd )$, $\bar S(\cd )$ are not all zero matrices. We will not follow the procedures in Subsection \ref{subsection-S-0}, but resort to a kind of  linear transformation techniques (referring to \cite{Wei-Yu-2021}).
Such  a linear transformation  can simplify  the general case to the  case  $S(\cdot )$ and $\bar S(\cdot)$ all being zero. Concretely, we introduce
\begin{equation}\label{equ-MFLQ-25}
 \mathbf{u}(\cdot):=u(\cdot)+R^1(\cdot)^{-1} S^1(\cdot) x^{(1)}(\cdot) + {R}^2(\cdot)^{-1} {S}^2(\cdot) x^{(2)}(\cdot).
\end{equation}
Then, the controlled system \eqref{state} and the cost functional \eqref{cost} can be rewritten as follows,
\begin{equation}\label{equ-MFLQ-26}
\left\{
\begin{aligned}
&\!\!\mathrm dx(s)=\sum_{\iota=1}^2\big[\mathscr{A}^{\iota}x^{(\iota)}+B^{\iota}\textbf{u} ^{(\iota)}\big] \, \mathrm ds   +\sum_{i=1}^d\sum_{\iota=1}^2\big[ \mathscr{C}_i^{\iota}x^{(\iota)}+D_i^{\iota}\textbf{u} ^{(\iota)}\big] \, \mathrm dW_i(s) \!\!\! \\
& \!\!\quad  +\sum_{j=1}^l\ds\int_E \sum_{\iota=1}^2\big[ \mathscr{M}^{\iota}_j( e) x^{(\iota)}+ N^{\iota}_j( e)\textbf{u}^{(\iota)}\big] \, \widetilde{\mu} _j\left(\mathrm ds, \mathrm de\right),  \quad  s\geq t, \!\\
&\!\!x(t)=x_{t},
\end{aligned}\right.
\end{equation}
and
\begin{equation}\label{equ-MFLQ-27}
\begin{aligned}
& \textbf{J}^{\widetilde K}\left(t, x_{t} ;  \textbf{u}(\cdot)\right) :=J^{K}\left(t, x_{t} ; u(\cdot)\right)=\dfrac{1}{2} \mathbb{E} \ds\int_{t}^{\infty}
\sum_{\iota=1}^2e^{2 \widetilde K s}\big[\big\langle   \mathscr{Q}^{\iota}  x^{(\iota)}, x^{(\iota)}\big\rangle
+\big\langle R^{\iota} \textbf{u} ^{(\iota)}, \textbf{u} ^{(\iota)}\big\rangle\big] \mathrm ds,
\end{aligned}
\end{equation}
where $\widetilde K=K$ and
\begin{equation}\label{bf-A-i}\begin{aligned}
&   \mathscr{A} ^{\iota}(\cdot):= A^{\iota}(\cdot)-B^{\iota}(\cdot)(R^{\iota})^{-1}(\cdot)S^{\iota}(\cdot),\q  \mathscr{C}^{\iota}_i(\cdot):= C_i^{\iota}(\cdot)-D_i^{\iota}(\cdot)(R^{\iota})^{-1}(\cdot)S^{\iota}(\cdot),\\
&  \mathscr{M}^{\iota}_j(\cdot,\cdot):=M^{\iota}_j(\cdot,\cdot)-N^{\iota}_j(\cdot,\cdot)(R^{\iota})^{-1}(\cdot)S^{\iota}(\cdot),\q \mathscr{Q}^{\iota}(\cdot):=Q^{\iota}(\cdot)-S^{\iota}(\cdot)^\top (R^{\iota})^{-1}(\cdot)S^{\iota}(\cdot),\\
& \mathscr{C}^{\iota} (\cdot)=(\mathscr{C}^{\iota}_1(\cdot)^\top,\dots,\mathscr{C}^{\iota}_l(\cdot)^\top)^\top,\q  \mathscr{M}^{\iota} (\cdot,\cdot)=(\mathscr{M}^{\iota}_1(\cdot,\cdot)^\top,\dots,\mathscr{M}^{\iota}_l(\cdot,\cdot)^\top)^\top.
\end{aligned}\end{equation}
For the control problem based on \eqref{equ-MFLQ-26} and \eqref{equ-MFLQ-27}, the corresponding Hamiltonian system is
\begin{equation}\label{Hamil-s-neq0}
\left\{
\begin{aligned}
&\!\! \mathrm dx (s)\!=\! \sum_{\iota=1}^2\big( \mathscr{A}^\iota x^{(\iota)}    -B^\iota  (R^\iota) ^{-1}\mathbf{\Lambda}^\iota[y^{(\iota)},z^{(\iota)},k^{(\iota)}]  \big) \, \mathrm ds  \\
&+\sum_{i=1}^d\sum_{\iota=1}^2\big( C_i^\iota x^{(\iota)}  -D_i^\iota (R^\iota)^{-1}\mathbf{\Lambda}^\iota[y^{(\iota)},z^{(\iota)},k^{(\iota)}] \big)\, \mathrm dW_i(s)\!\! \\
&+\sum_{j=1}^l\ds\int_E \sum_{\iota=1}^2\big( \mathscr{M}_j^{\iota} (e)x^{(\iota)}  -N_j^\iota (e)(R^\iota)^{-1}\mathbf{\Lambda}^\iota[y^{(\iota)},z^{(\iota)},k^{(\iota)}] \big) \, \widetilde{\mu}_j\left(\mathrm ds, \mathrm de\right), \quad s\geq t,\\
&\!\!\mathrm dy (s)\! =\!-\sum_{\iota=1}^2\big(  \mathscr{Q}^\iota x^{(\iota)} +(2\widetilde K I+ \mathscr{A}^\iota)^{\top}y^{(\iota)}
+\sum_{i=1}^d(\mathscr{C}^\iota)^{\top} z^{(\iota)}  +\int_E   \mathscr{M}^{\iota}(e)^\top\varrho(\mathrm de)  k^{(\iota)}(e) \big) \, \mathrm d s  \\
& +\sum_{i=1}^d z _i(s) \, \mathrm dW_i(s) + \ds\sum_{j=1}^l\int_E k_j (s, e) \, \widetilde{\mu}_j \left(\mathrm ds, \mathrm de\right),   \quad s\geq t,\\
&\!\!x (t)\!=\!x_{t}.
\end{aligned}\right.
\end{equation}
Note that, in \eqref{equ-MFLQ-27}, the cross terms disappear. Therefore, we can apply the obtained result in Subsection \ref{subsection-S-0} to derive the wellposedness of \eqref{Hamil-s-neq0}, which is similar to \eqref{FBSDE-S=0} by considering the notations \eqref{bf-A-i}. To  formulate the wellposedness result, we introduce  the following new notations,
\begin{equation}\left\{
\begin{aligned}
&\!\!\kappa'_1:=-\frac{1}{2}\mathop{\sup}\limits_{s \in[t, \infty)}\lambda_{\max }\(  \mathscr{A}^2(s) +
\mathscr{A}^2(s) ^{\top} \),\\
&\!\!\kappa'_2:=-\frac{1}{2}\mathop{\sup}\limits_{s \in[t, \infty)}\lambda_{\max }\( \mathscr{A}^1(s)+\mathscr{A}^1(s)^{\top}+ \mathscr{C}^1(s)^{\top}\mathscr{C}^1(s)
+\int_\mathcal{E}\mathscr{M}^1(s, e)^{\top}\varrho(\mathrm de)\mathscr{M}^1(s, e) \),\\
&\!\! \kappa' :=\min\{\kappa'_1,\kappa'_2 \}.
\end{aligned}\right.
\end{equation}

\begin{proposition}\label{Pro-Sneq0}\sl
Let {\bf{(H$_{1}$)}}, {\bf{(H$_{2}$)}} and {\bf (PD)} hold,  and $\widetilde{K}<\widetilde{\kappa}$.
Then,  there exists some constant $\widetilde{\mathbf{k}}>0$, such that for any $\displaystyle \widetilde K\in \[\frac{\widetilde \kappa_1-\widetilde\kappa_2}{2}, \frac{\widetilde\kappa_1-\widetilde\kappa_2}{2}+\widetilde{\mathbf{k}}\)$, for any $\left(t, x_{t}\right) \in [0, \infty) \times L_{\mathcal{F}_{t}}^{2}\left(\Omega ; \mathbb{R}^{n}\right)$, the infinite horizon MF-FBSDE with jumps  \eqref{Hamil-s-neq0} (or \eqref{Hamil} with $K=\widetilde K$) admits a unique solution $\theta(\cdot) \in \cL_{\mathbb{F}}^{2, \widetilde{K}}(t, \infty)$. Further, $u^*(\cdot) $ given by \eqref{optimal-u}  lying in $ \mathcal{U} ^{\ti{K}}[t, \infty)$ is a unique optimal control of Problem (LQ) at $(t, x_t)$.

\end{proposition}

%
%

By comparing  Proposition \ref{Th-FBSDE-1} and Proposition \ref{Pro-Sneq0},
 we find  the solutions to the equations \eqref{FBSDE-S=0} and \eqref{Hamil-s-neq0} lie in the different spaces.
Moreover,
the cross terms  $S(\cd )$, $\bar S(\cd )$ affect  the existing spaces of the open-loop optimal controls of Problem (MF-LQ).
This   does not  happen  in the finite horizon optimal control problems, referring to \cite{WYY-2019}.
 Such a distinction has been discovered for a kind of  infinite horizon LQ control problems, which corresponds to   the mean-field and jumps terms disappearing in Problem (MF-LQ), referring to \cite{Wei-Yu-2021}.   Examples 4.5 and 4.8 therein  illustrated the distinction intuitively. An illustrative example is also  provided here   to support and better understand Propositions \ref{Th-FBSDE-1} and \ref{Pro-Sneq0}. For simplified,  we set $n=m=d=1$ and ignore the Poisson random measure in the following example.

 \begin{example}
Consider  the following MF-SDE,
\begin{equation}\label{equ-MFLQ-30}
\left\{
\begin{aligned}
&\!\!\mathrm dx(s)=\[-\big(a(s)+2\rho\big)x(s)+ a(s)\mathbb{E}{[x(s)]}-\big(a(s)+\frac32\rho\big)u(s)+a(s)\mathbb{E}{[u(s)]}\] \, \mathrm ds\\
&\hskip1.1cm    + \sqrt{2\big(a(s)+\rho\big)}\[x(s)+\mathbb{E}{[x(s)]}+ u(s)+\mathbb{E}{[u(s)]}\] \, \mathrm dW(s) ,\qq s\geq 0, \!\!\!\!\! \\
&\!\!x(0)=x_{0}\neq0,
\end{aligned}\right.
\end{equation}
where $\rho$ is a positive constant,
$a(\cdot)$ is   a deterministic bounded
function  satisfying $a(\cd)+\rho > 0$,
 and the control process $u(\cdot)\in \mathcal{U}_{ad}^K[0,\infty)$.
Our aim is to minimize the following cost functional
\begin{equation}\label{equ-MFLQ-29}
\begin{array}{lll}
&\!\!\!J^{K}\left(0,x_{0} ; u(\cdot)\right) =\dfrac{1}{2} \mathbb{E} \ds\int_{0}^{\infty} e^{ 2K s}\big[(x(s)+u(s))^2+(\mathbb{E}[ x(s)]+\mathbb{E}[ u(s)] )^2\big]\,\mathrm ds.
\end{array}
\end{equation}
Corresponding to the  setting in previous sections,  we  know $\kappa_1=2\rho$, $\kappa=\kappa_2=\rho$, and  the parameter   $K<\rho$.
Moreover, the associated Hamiltonian system   reads (ignoring the variable $s$),
\begin{equation}\label{equ-MFLQ-31}
\left\{
\begin{aligned}
&\!\!\mathrm dx(s)\!=\!\[\!\!-\frac{\rho}{2} x-(a+\frac{3}{2}\rho)^{2}y
+\(\!a^2+3\rho a+\frac{9}{8}\rho ^{2}\! \)\mathbb{E}{[y]}+\sqrt{2(a+\rho)}\(\!(a+\frac{3}{2}\rho)z
-a\mathbb{E}{[z]} \!\)\] \, \mathrm ds\\
&\hskip1.2cm  + \sqrt{2(a+\rho)}\Big[\!(a+\frac{3}{2}\rho)y-a\mathbb{E}{[y]}-\sqrt{2(a+\rho)}\big(z+\mathbb{E}{[z]}\big)\!\Big]\, \mathrm dW(s),  \!\!\!\!\! \\
&\!\!\mathrm dy(s)=-\(2K- \frac{\rho}{2}  \)y \, \mathrm ds + z \, \mathrm dW(s)
 , \qq s\geq 0, \!\!\!\!\! \\
&\!\!x(0)=x_{0}.
\end{aligned}\right.
\end{equation}

Now we claim that \eqref{equ-MFLQ-31} does not admit a solution in $L_\mathbb{F}^{2,K}(0,\infty;\mathbb{R}^3)$ with $K\in\left[\frac\rho2,  \rho\right)$. In fact, $(0,0)\in L_\mathbb{F}^{2,K}(0,\infty;\mathbb{R}^2)$ is the unique solution to BSDE in \eqref{equ-MFLQ-31}.
Then,
\begin{equation}\label{equ-MFLQ-32}
\left\{
\begin{aligned}
&\!\!\mathrm d x (s)= -\frac{\rho}{2} x (s) \, \mathrm ds, \q s\geq 0,  \\
&\!\!x (0)=x_{0}\neq 0,
\end{aligned}\right.
\end{equation}
 whose unique  solution is $ x (s)=x_0 e^{-\frac{\rho}{2}  s},$ $ s\geq 0$.  Further, $x(\cd)\in  L_\mathbb{F}^{2,K}(0,\infty;\mathbb{R})$ implies
 $x_0 = 0$, which contradicts  $x_0\neq 0$.
Therefore,  FBSDEs \eqref{equ-MFLQ-31} admits no solution in $L_\mathbb{F}^{2,K}(0,\infty;\mathbb{R}^3)$ with $K\in\left[\frac\rho2,  \rho\right)$.

By checking Proposition \ref{Th-FBSDE-1},  we find it fails for \eqref{equ-MFLQ-31}, which is due to $S(\cdot)\neq 0$ and  $\bar S(\cdot)\neq 0$.  Further, Lemma \ref{Le-MFLQ-1} implies that the control problem formulated by  \eqref{equ-MFLQ-29} and \eqref {equ-MFLQ-30}  does not have the optimal control in  $\mathcal{U}_{ad}^K[0, \infty )$  with $K\in\left[\frac\rho2,  \rho\right)$.

 \medskip

Next,   we  resort to Proposition  \ref{Pro-Sneq0} to tackle with the nonzero $S(\cdot)$ and  $\bar S(\cdot)$.
Applying  the linear transformation \eqref{equ-MFLQ-25} to the controlled system \eqref{equ-MFLQ-30} and the cost functional \eqref{equ-MFLQ-31}, we get
\begin{equation}\label{equ-MFLQ-34}
\left\{
\begin{aligned}
&\!\!\mathrm dx(s)=\[-\frac{\rho}{2}  x(s)-(a(s)+\frac{3\rho}{2})\mathbf{u}(s)+a(s)\mathbb{E}{[\mathbf{u}(s)]}\] \, \mathrm ds\\
&\hskip1.1cm    + \sqrt{2(a(s)+\rho)}\[\mathbf{u}(s)+\mathbb{E}{[\mathbf{u}(s)]}\] \, \mathrm dW(s) ,\qq s\geq 0, \!\!\!\!\! \\
&\!\!x(0)=x_{0}\neq 0,
\end{aligned}\right.
\end{equation}
and
\begin{equation}\label{equ-MFLQ-33}
\begin{array}{lll}
\mathbf{J}^{\widetilde K}\left(0, x_{0} ; \mathbf{u}(\cdot)\right)
=\dfrac{1}{2}\mathbb{E} \ds\int_{0}^{\infty} e^{2 \widetilde K s}\big[\langle \mathbf{u}(s), \mathbf{u}(s)\rangle
+\langle \mathbb{E}[\mathbf{u}(s)], \mathbb{E}[\mathbf{u}(s)]\rangle\big]\mathrm ds,\q \widetilde K=K.
\end{array}
\end{equation}
The    Hamiltonian system corresponding to the control problem with \eqref{equ-MFLQ-34} and \eqref{equ-MFLQ-33}  is in fact still \eqref{equ-MFLQ-31} with $K=\widetilde K$.
In this framework,  $\widetilde{\kappa}=\widetilde{\k}_1=\widetilde{\k}_2=\frac\rho 2$.
Applying Proposition \ref{Pro-Sneq0}, there exists some constant $\widetilde {\mathbf{k}}>0$, such that MF-FBSDEs \eqref{equ-MFLQ-31} possesses a unique solution $\theta(\cdot)=(x(\cdot), y(\cdot), z(\cdot))^{\top}\in L_{\mathbb{F}}^{2,  \widetilde  K}(0, \infty ; \mathbb{R}^3)$ with $  \widetilde K\in \left[0,\widetilde{\mathbf{k}} \wedge \frac\rho 2\right)$.
  It is exactly that $\theta(s)=(x_0 e^{-\frac{\rho}{2} s}, 0,0)^{\top}  \in L_{\mathbb{F}}^{2,   \widetilde  K}(0, \infty ; \mathbb{R}^3)$ with $  \widetilde K\in \left[0,\widetilde{\mathbf{k}} \wedge \frac\rho 2\right)$.
Moreover, from Lemma \ref{Le-MFLQ-1}, when $  \widetilde  K\in \left[0,\widetilde{\mathbf{k}} \wedge \frac\rho 2\right)$,  $u^\ast(\cdot)=-x_0 e^{-\frac{\rho}{2} \cdot} \in \mathcal{U}_{ad}^ {\widetilde K}[0,\infty)$ is  a unique  optimal control of \eqref{equ-MFLQ-29} and \eqref {equ-MFLQ-30}.

 \end{example}


\begin{thebibliography}{99}

%
%
%
%
%

 \bibitem{AO-2014} N. Agram and B. {\O}ksendal, \it Infinite horizon optimal control of forward-backward stochastic
differential equations with delay, \sl J. Comput. Appl. Math., \rm 259, (2014), Part B, 336-349.


\bibitem{DM-2019} R. Deepa and P. Muthukumar, \it Infinite horizon optimal control of mean-field delay system
with semi-Markov modulated jump-diffusion processes, \sl J. Anal., \rm 27(2), (2019), 623-641.


\bibitem{GT-2008} G. Guatteri and G. Tessitore, \it Backward stochastic Riccati equations and infinite horizon
L-Q optimal control with infinite dimensional state space and random coefficients, \sl Appl.
Math. Optim., \rm 57, (2008),   207-235.


\bibitem{HOP-2013} S. Haadem, B. {\O}ksendal and F. Proske, \it Maximum principles for jump diffusion processes
with infinite horizon,  \sl Automatica J. IFAC, \rm 49(7), (2013),    2267-2275.



\bibitem{HLY-2015}
J. Huang, X. Li and  J. Yong, \it A linear-quadratic optimal control problem for
mean-field stochastic differential equations in
infinite horizon, \sl Math. Control Relat. Fields,
\rm 5(1),  (2015), 97-139.

\bibitem{HP-1995} Y. Hu and S. Peng, \it Solution of forward-backward stochastic differential equations, \sl Probab. Theory Related Fields, \rm 103(2), (1995),  273-283.


 \bibitem{LSY-2021} X. Li, J. Shi and  J. Yong, \it Mean-field linear-quadratic stochastic differential
games in an infinite horizon. \sl ESAIM Control Optim. Calc. Var.,
\rm 27, (2021), 81.
%

\bibitem{MWY-2021}H. Mei, Q. Wei and  J. Yong, \it Optimal ergodic control of linear stochastic differential equations with quadratic cost functional having indefinite weights, \sl  SIAM J. Control Optim., \rm 59(1), (2021), 584-613.

\bibitem{OV-2017}  C. Orrieri and P. Veverka, \it Necessary stochastic maximum principle for dissipative systems
on infinite time horizon, \sl ESAIM Control Optim. Calc. Var., \rm 23(1), (2017), 337-371.



\bibitem{Peng-Shi}
S. Peng and Y. Shi, \it Infinite horizon forward-backward stochastic differential equations, \sl Stoch.
Process. Appl., \rm 85, (2000),   75--92.

\bibitem{Peng-Wu-1999} S. Peng and Z. Wu, \it Fully coupled forward-backward stochastic differential equations and applications to optimal control, \sl SIAM J. Control Optim., \rm 37(3), (1999),   825-843.

\bibitem{PZ-2021} J. Pu and Q. Zhang, \it Constrained stochastic LQ optimal control problem with random coefficients
on infinite time horizon, \sl  Appl. Math. Optim., \rm 83, (2021),   1005-1023.


\bibitem{RZ-2000} M. Rami and X. Zhou, \it Linear matrix inequalities, Riccati equations, and indefinite stochastic linear quadratic controls, \sl IEEE Trans. Automat. Control, \rm 45, (2000),   1131-1143.



\bibitem{RM-2022} A. Roubi and M.  Mezerdi, \it Necessary and sufficient conditions in optimal control of mean-
field stochastic differential equations with infinite horizon, \sl Random Oper. Stoch. Equ., \rm 30(3),
(2022),   183-195.



\bibitem{Sun-Yong-2018} J. Sun and J. Yong, \it Stochastic linear quadratic optimal control problems in infinite horizon, \sl Appl. Math. Optim., \rm 78, (2018),   145-183.

\bibitem{Shi-Zhao-2020}
Y. Shi and H. Zhao, \it Forward-backward stochastic differential equations on infinite horizon
and quasilinear elliptic PDEs, \sl J. Math. Anal. Appl., \rm 485(1), (2020),   123791.

\bibitem{Tian-Yu-2020}
R.  Tian and Z. Yu, \it  Mean-field type FBSDEs under domination-monotonicity conditions and application to LQ problems, \sl
SIAM J. Control Optim. \rm 61(1), (2023),  22-46.

\bibitem{Wei-Yu-2021}
Q. Wei and Z. Yu, \it Infinite horizon forward-backward SDEs and
open-loop optimal controls for stochastic
linear-quadratic problems with random
coefficients, \sl SIAM J. Control Optim., \rm 59(4), (2021), 2594-2623.


\bibitem{WYY-2019}
Q. Wei, J. Yong and Z. Yu, \it Linear quadratic stochastic optimal control problems with operator coefficients: open-loop solutions, \sl ESAIM Control Optim. Calc. Var.  \rm 25, (2019),   17.



\bibitem{Yong-2013} J. Yong, \it Linear-quadratic optimal control problems for mean-field stochastic differential equations, \sl SIAM J. Control Optim., \rm 51, (2013), 2809-2838.

\bibitem{Yu-2017} Z. Yu, \it Infinite horizon jump-diffusion forward-backward stochastic
differential equations and their application to backward
linear-quadratic problems, \sl ESAIM Control Optim. Calc. Var., \rm 23, (2017), 1331-1359.




\end{thebibliography}
\end{document}